\documentclass[12pt, oneside, reqno]{amsart}
\usepackage[left=3.25cm,right=3.25cm,top=2.5cm,bottom=2.5cm,includeheadfoot]{geometry}
\usepackage[utf8]{inputenc}
\usepackage{cite}
\usepackage{url}
\usepackage{amsmath}
\usepackage{amssymb}
\usepackage{amsfonts}

\theoremstyle{plain}
\newtheorem{thm}{Theorem}[subsection]
\newtheorem{lem}[thm]{Lemma}

\newtheorem{prop}[thm]{Proposition}
\theoremstyle{definition}
\newtheorem{defn}[thm]{Definition} 

\newtheorem{rem}[thm]{Remark}

\numberwithin{equation}{section}

\newcommand{\lin}{\textrm{span}}
\begin{document}

\title[DDSDDEs in finite and infinite dimensions]{Distribution-Dependent Stochastic Differential Delay Equations in finite and infinite dimensions}

\author{Rico Heinemann*}

\thanks{*Faculty of Mathematics, Bielefeld University, Universit\"atsstrasse 25, 33615 Bielefeld, Germany\\
Email: rico.heinemann@t-online.de}



\maketitle

\begin{abstract}
We prove that distribution dependent (also called McKean--Vlasov) stochastic delay equations of the form
\begin{equation*}
\text{d}X(t)= b(t,X_t,\mathcal{L}_{X_t})\text{d}t+ \sigma(t,X_t,\mathcal{L}_{X_t})\text{d}W(t)
\end{equation*}
have unique (strong) solutions in finite as well as infinite dimensional state spaces if the coefficients fulfill certain monotonicity assumptions.
\end{abstract}

\section*{Introduction}
The aim of this paper is to study the existence and uniqueness of Distribution-Dependent Stochastic Differential Delay Equations (DDSDDE's) in finite and infinite dimensional state spaces in the variational framework.
A DDSDDE has the form
\begin{equation*}
\text{d}X(t)= b(t,X_t,\mathcal{L}_{X_t})\text{d}t+ \sigma(t,X_t,\mathcal{L}_{X_t})\text{d}W(t),
\end{equation*}
where $W$ is a standard $\mathbb{R}^d$-valued Wiener process in the finite dimensional case and a cylindrical $Q$-Wiener process with $Q=I$ in a separable Hilbert space in the infinite dimensional case.
$X_t$ denotes the \textit{delay} or \textit{segment} of $X$ at time $t$.
$X_t$ takes values in a path-space and is defined as $X_t(\theta):=X(t+\theta)$, $\theta \in[-r_0,0]$,
whereby $r_0>0$ is fixed.
$\mathcal{L}_{X_t}$ denotes the \textit{law} of $X_t$. \\
Recently there has been an increasing interest in this type of equations as well as in classical distribution-dependent SDE's (DDSDE's) - also referred to as McKean-Vlasov SDEs - , i.e. equations of the form
\begin{align*}
\text{d}X(t)= b(t,X(t),\mathcal{L}_{X(t)})\text{d}t+ \sigma(t,X(t),\mathcal{L}_{X(t)})\text{d}W(t),
\end{align*} 
see for instance \cite{CD18}, \cite{1803.09320}, \cite{1802.03974}, \cite{MP}, \cite{1805.01682}, \cite{1805.01654}, \cite{1603.02212}, \cite{Roeck1}, \cite{Roeck2}, \cite{WangH}  or \cite{ODLP} as well as the references therein.
Clearly, SDDE's can be viewed as a sub-class of DDSDDE's.\\
A first existence and uniqueness result under monotonicity conditions for distribution-dependent SDE's without delay was published by Wang in 2018 (see reference \cite{ODLP}).
Wang's idea was carried over to the case with delay by Huang, Röckner and Wang in \cite{MP}.
\cite{MP} and \cite{ODLP} are the main reference for the second chapter of this paper.\\

A main motivation to study solutions of DDSDDE's is their relation to solutions of \textit{non-linear Fokker-Planck Kolmogorov equations} (FPKE's).
Whenever coefficients $b$ and $\sigma$ are given one can define the following differential operator from $C_0^{\infty}(\mathbb{R}^d)$ to the set of all Borel-measurable real-valued functions on $\mathcal{C}$:
\begin{align*}
\left(L_{t,\mu}f\right)(\xi):=
\sum _{i=1}^d b_i(t,\xi,\mu)(\partial _i f)(\xi (0))
+\frac{1}{2}\sum _{i,j=1}^d (\sigma \sigma^*)_{i,j}(t,\xi,\mu)(\partial _i\partial _j f)(\xi (0)),
\end{align*}
$t \geq 0$, $\mu \in \mathcal{P}_2(\mathcal{C})$, $f \in C_0^{\infty}(\mathbb{R}^d)$ and $\xi \in \mathcal{C}:=C([-r_0,0];\mathbb{R}^d)$, where $\mathcal{P}_2(\mathcal{C})$ denotes the set of all probability measues on $\mathbb R^d$ with finite second moments.
By Itô's formula one can show that if $(X(t))_{t\geq -r_0}$ is a solution of our DDSDDE,
$\mu _t:=\mathcal{L}_{X_t}$ solves the corresponding FPKE
\begin{align*}
\partial \mu (t):=L_{t,\mu _t}^* \mu _t, 
\end{align*}
where $\mu (t):=\mathcal{L}_{X(t)}= \text{law of } X(t)$.
Here, we call a continuous mapping $\mu :\mathbb{R}_+ \rightarrow \mathcal{P}_2(\mathcal{C})$ a solution of the FPKE, if
\begin{align*}
\int_0^t \int_{\mathcal{C}} \vert L_{s,\mu _s} f \vert \text{d}\mu_s\text{d}s
<\infty,
\end{align*}
and
\begin{align*}
\int _{\mathbb{R}^d} f \text{d}\mu(t)
= \int _{\mathbb{R}^d} f \text{d}\mu(0) 
+ \int_0^t \int_{\mathcal{C}} \left( L_{\mu _s,s} f\right)\text{d}\mu _s\text{d}s
\end{align*}
for all $t \geq 0$ and  $f \in C_0^{\infty}(\mathbb{R}^d)$.
For more details of the relation between Fokker-Planck equations and DDSDDE's see for instance \cite{Roeck2} or \cite{MP}.
Since this paper focuses on the existence and uniqueness of DDSDDE's, we are not going to further investigate this relation and just note that an existence and uniqueness result for FPKE's could be deduced from our existence and uniqueness result in Chapter 2 as in \cite[Chapter 2]{MP}.
For more information about FPKE see for instance \cite{FPPK}.\\

A first result for the existence and uniqueness of solutions to DDSDDE's in finite dimensions was proved by Huang, Röckner and Wang in 2017 (see \cite{MP}).
The main novelty of this paper, compared to \cite{MP}, is that we also prove an existence and uniqueness result in infinite dimensions (Theorem \ref{MT3}), i.e.~we replace $\mathbb R^d$ in \cite{MP} by a separable real Hilbert space.
To be able to do this we prove another finite dimensional result (Theorem \ref{MT2}) under assumptions which are better suited for the generalization to infinite dimensions as the conditions presented in \cite{MP}.
Moreover, our proof of the finite dimensional result replaces the iteration procedure in \cite{MP} by a fixed point argument, which turns out to be technically easier and more conceptual. \\

Next let us give a brief overview of the content of this paper.
The first chapter introduces some tools and notations, which are necessary to understand this paper.
In the second chapter, we prove that for every initial condition $\psi \in \mathcal{C}:=C([-r_0,0];\mathbb{R}^d)$ a DDSDDE has a unique solution, if certain monotonicity, coercivity, growth and continuity assumptions are fulfilled (Theorem \ref{MT2}).
To be able to do this, we define precisely what a \textit{solution} is (see Definition \ref{DefWsol} and Definition  \ref{DefSol}).
While the main idea of our proof is similar to the proof in \cite{MP}, i.e. we deduce existence of a solution to DDSDDE's from a result for stochastic differential delay equations (SDDE's),
we assume different conditions on the coefficients (see (H1) to (H4) in Chapter 3), which are better suited for the infinite dimensional case and use Banach's fixed-point theorem instead of \textit{iterating in distribution},
i.e. approximating the solution to the DDSDDE by solutions of SDDE's.
The conditions on $b$ and $\sigma$ are chosen in such a way,
that given any fixed continuous, adapted, $\mathbb{R}^d$-valued process $(X(t))_{t\geq -r_0}$ with $\mathbb{E}\left[\sup _{t \in [-r_0,T]} \vert X(t) \vert ^2 \right] < \infty$ for all $T>0$, the classical SDDE
\begin{align*}
\text{d}\left(\Lambda X\right)(t)=b(t,(\Lambda X)_t,\mathcal{L}_{X_t})\text{d}t+ \sigma(t,(\Lambda X)_t,\mathcal{L}_{X_t})\text{d}W(t),
\end{align*}
has a unique solution $\Lambda X$, fulfilling the initial condition $ \Lambda X (0)= \psi $ for given $\psi \in \mathcal{C}$.
It is clear, that $X$ is a solution of our DDSDDE, if $\Lambda X = X$.
Using the Banach fixed-point theorem, we show that there exists exactly one $X$, such that  $\Lambda X = X$ (see Lemma \ref{ExY} and Lemma \ref{KontraLem}).
In addition to existence and path-wise uniqueness, we also prove weak uniqueness.
The weak uniqueness is derived from the Yamada-Watanabe Theorem for SDDE's.\\

The third chapter contains the main novelty of this paper, as we prove an existence and uniqueness result for DDSDDE's in infinite dimensions.
That is, we replace $\mathbb{R}^d$ by a separable Hilbert space $H$, more precisely an appropriate Gelfand triple $(V,H,V^*)$.
Chapter 3 is an extension of the fourth chapter in \cite{Liu2015} to DDSDDE's, i.e. we work in the \textit{variational framework} and use a \textit{Galerkin approximation} to deduce the infinite dimensional result from the finite dimensional result.

%\newpage
\section{Preliminaries}
This chapter introduces some notations and results needed for the formulation and understanding of the rest of this paper, like the Wasserstein distance and an existence and uniqueness result for SDDE's.
All results are given without proof, since they are not the actual topic of this paper. \\
In addition to contents of this chapter, knowledge about measure and integration theory (c.f. \cite{MIT}), functional analysis (c.f. \cite{alt} or \cite{Yosida}), probability theory (c.f. \cite{BWT}, \cite{Doob} or \cite{Str}),
stochastic integration theory (c.f. \cite{Liu2015} or \cite{STInt}) as well as stochastic differential equations (c.f. \cite{DaPrato} or \cite{Liu2015}), is necessary to understand this paper.

\subsection{Notations}
First of all, let us fix some notations.
As usual we denote $\mathbb{N}$,  $\mathbb{Q}$ and $\mathbb{R}$ for the set of all natural, rational and real numbers, respectively. 
For $d \in \mathbb{N}$, $\mathbb{R}^d$ denotes the $d$-dimensional euclidean space, $\langle \cdot,\cdot\rangle$ the inner product and $\vert\cdot\vert$ the corresponding norm.
If $m \in \mathbb{N}$ is another natural number, $\mathbb{R}^{d\times m}$ denotes the space of all $d\times m$-matrices.
If $A$ is an arbitrary set, we write $1_A$ for its indicator function.
For $s,t \in \mathbb{R}$ we define $s \vee t := \max (s,t)$ and $s \wedge t := \min (s,t)$.
Like usual, for $a,b \in \bar{\mathbb{R}}:=\mathbb{R}\cup \lbrace - \infty \rbrace \cup \lbrace + \infty \rbrace$, $(a,b):=\lbrace x \in \bar{\mathbb{R}} : \ a < x < b \rbrace$ denotes the open interval, $[a,b]$ denotes the closed interval, $[a,b)$ denotes the left-closed interval and $(a,b]$ denotes the left-open interval.
We call $(\Omega,\mathcal{F},(\mathcal{F}_t)_{t\geq -r_0},P)$ a \textit{stochastic basis},
if $(\Omega,\mathcal{F},P)$ is a \textit{complete probability space} and $(\mathcal{F}_t)_{t\geq -r_0}$ is a \textit{normal filtration}.
If $(X, \Vert \cdot \Vert _X )$ is a Banach space, we denote $X^*$ for the dual space of $X$
and $\mathcal{B}(X)$ for the Borel sigma-algebra on $X$.
For $x \in X$ and $x^*\in X^*$ we define ${}_{X^*}\langle x^*,x\rangle _X=x^*(x)$ as the \textit{dualization} between $X$ and $X^*$.
$I_X$ or, if it is clear on which space we are working, $I$ denotes the identity operator.
If $X$ and $Y$ are Banach spaces, we denote $L(X,Y)$ for the Banach space of all bounded linear operators from $X$ to $Y$
equipped with the standard operator norm.
Moreover, if $U$ and $H$ are Hilbert spaces we denote $L_2(U,H)$ for the space of all Hilbert-Schmidt operators from $U$ to $H$ equipped with the usual norm
$\Vert T \Vert _{\text{HS}} = \left(\sum _{n=1}^{\infty} \Vert Te_n \Vert_H^2 \right)^{\frac{1}{2}}.$

\subsubsection{Path spaces}
As we will see in section 1.3 as well in chapter 2 and 3, the coefficients of a stochastic delay equation are defined on path spaces, i.e. spaces of functions.
This subsection introduces the spaces of functions which are needed in this paper.\\
Throughout this paper, whenever $(X,\Vert \cdot \Vert _X)$ is a Banach space, $p\geq 2$ and $r_0>0$ fixed, we use the following notations:\\
If $(E,d)$ is a metric space, $C(E;X)$ denotes, like usual, the set of all continuous functions from $E$ to $X$.
If $(S,\mathcal{A},\mu)$ is a measure space, $L^p(S,\mathcal{A},\mu;X)$ denotes the usual $L^p$-space (c.f. \cite{MIT}, \cite[Chapter V.5]{Yosida} or \cite[Appendix A]{Liu2015}).
If it is clear which sigma algebra $\mathcal{A}$ or which measure $\mu$ is used, we might for simplicity just denote  $L^p(S,\mu;X)$,  $L^p(S,\mathcal{A};X)$ or  $L^p(S;X)$, respectively.
In the case that $S \in \mathcal{B}(\mathbb{R}^d)$, it is always $\mathcal{A}=\mathcal{B}(S)$ and $\mu$ is the Lebesgue measure.
 $\mathcal{C}(X):=C({[-r_0,0]};X)$
equipped with the uniform norm $\Vert \xi \Vert _{\infty} := \sup _{\theta \in [-r_0,0]}\Vert \xi (\theta)\Vert _X,$
 $\mathcal{C}_{\infty}(X)=C({[-r_0,\infty)};X)$,
 equipped with the metric
 $d(\xi,\eta):=\sum _{k \in \mathbb{N}} 2^{-k} \big(\sup _{t [- r_0,k]} \Vert \xi(t)-\eta(t)\Vert _X \wedge 1\big)$,
$L^p_X:=L^p({[-r_0,0]};X)$
equipped with the standard $L^p$-norm $ \Vert \xi \Vert _{L^p_X}^p:= \int _{-r_0}^0 \Vert \xi (z) \Vert^p _X\text{d}z$.
In the case $X=\mathbb{R}^d$, for some $d \in \mathbb{N}$, we just write $\mathcal{C}, \mathcal{C}_{\infty}$ and $L^p$, respectively.

Moreover, whenever $(Y,\Vert \cdot \Vert _Y)$ is another Banach space with $X \subset Y$ continuously and $I \subset \mathbb{R}$ is an interval, we define
\begin{align*}
C(I;Y)\cap L^p(I;X)
:=\bigg\lbrace \xi \in  C(I;Y): &\ \exists \bar{\xi}:I \rightarrow X \ 
\mathcal{B}(I)/ \mathcal{B}(X)\text{-measurable such }
\\ & \text{that }\bar{\xi}=\xi \text{ d}t-a.e. 
\text{ and } \int _I \left\Vert \bar{\xi}(t)\right\Vert^p\text{d}t<\infty \bigg\rbrace
\end{align*}
and
\begin{align*}
C(I;Y)\cap L^p_{loc}(I;X)
:=\bigg\lbrace \xi \in C(I;Y): &\ \exists \bar{\xi}:I \rightarrow X 
\text{ such that } \bar{\xi}=\xi \text{ d}t-a.e.
\\ &\text{and } \int _{I^{'}} \left\Vert \bar{\xi}(t)\right\Vert^p\text{d}t<\infty
\ \forall I^{'} \subset I \text{ compact.} \bigg\rbrace.
\end{align*}
Obviously $C(I;Y)\cap L^p(I;X)=C(I;Y)\cap L^p_{loc}(I;X)$ if $I$ is compact
and $E:=L^p(I,Y)$, $C(I;Y)\cap L^p(I;X)$ is a Banach space under the norm 
$\Vert \cdot \Vert_{C(I;Y)\cap L^p(I;X)}:=\Vert \cdot \Vert_{C(I;Y)}+\Vert \cdot \Vert_{ L^p(I;X)}$, in the case that $I$ is compact.

\subsubsection{Segments of functions}
The main difference between stochastic delay differential equations and classical SDE's is - as the name already suggests - that the coefficients depend on the \textit{delay} of $X$ at time $t$ instead of the value of $X$ at time $t$.
Therefore we have to define precisely what the \textit{segment} or \textit{delay} of a function is.
We do this similar to similar to \cite[Chapter 2]{MP}.
Let $X$ be a Banach space and $r_0>0$ fixed.
For $t\geq 0$ define the map
\begin{align*}
\pi _t \colon \mathcal{C}_{\infty}(X) \longrightarrow \mathcal{C}(X)
\end{align*}
by $(\pi _t f)(\theta)=f(t+\theta)$, $\theta \in [-r_0,0]$.
In the following we will denote $f_t:=\pi _t f$, for $t\geq 0$.
\begin{rem}
\label{SegStet}
Note that $[0, \infty) \ni t \mapsto f_t$ is an element in $C([0, \infty),\mathcal{C}(X))$.
\end{rem}

\subsection{p-th order probability measures and Wasserstein distance} 
Since we want to study stochastic differential equations where the coefficients also depend on the distribution of the solution, we need to have a measure for the distance between to probability measures to be able to formulate monotonicity assumptions on our coefficients.
The \textit{Wasserstein distance} is the most important tool to do that.
The main references for this section are \cite{Ambro} and \cite{Villani2008}.
Throughout this section, let $(X, \Vert \cdot \Vert _X )$ be a separable Banach space, $\mathcal{B}(X)$ the Borel sigma-algebra on $X$ and $\mathcal{P}(X)$ the set of all probability measures on $(X, \mathcal{B}(X))$.
\begin{defn}
Let $p \geq 1$. \textit{The class of probability measures of p-th order} is defined as
\[ \mathcal{P}_p(X):=\left\lbrace \mu \in \mathcal{P}(X) : \  \text{  }  \mu(\Vert \cdot \Vert _X^p):= \int_X \Vert x \Vert _X^p \mu(\text{d}x) < \infty \right\rbrace.
\]
\end{defn}
On $\mathcal{P}_p(X)$ we can define the following metric:
\begin{defn}
For $\mu, \nu \in \mathcal{P}_p(X)$ define \textit{the p-th Wasserstein distance} as
\[
\mathbb{W}_p^X(\mu,\nu):=\underset{\gamma \in \Gamma(\mu,\nu)}{\inf}
\left(\int_{X\times X} 
\Vert x-y \Vert _X^p \gamma(\text{d}x,\text{d}y)\right)^{\frac{1}{p}}.
\]
Here $\Gamma(\mu,\nu)$ denotes the set of all couplings of $\mu$ and $\nu$, e.g.
\[
\Gamma(\mu,\nu):=\left\lbrace \gamma \in \mathcal{P}(X \times X): \ \text{  } \gamma \circ \pi _x ^{-1} = \mu \text{ and } \gamma \circ \pi _y ^{-1} = \nu \right\rbrace,
\]
where $\pi _x(x,y):=x$ and $\pi _y(x,y):=y$, $(x,y) \in X \times X$, are the standard projections.
(Here $X\times X$ is equipped with the $\sigma$-field generated by the projections.)
\end{defn}
\begin{prop}
$(\mathcal{P}_p(X),\mathbb{W}_p^X)$ is a polish space, e.g. a separable, complete metric space.
\end{prop}
\begin{proof}
See \cite[Proposition 7.1.5]{Ambro}
\end{proof}

\subsection{Spaces of measure-valued functions}
As we will see in the next chapter, the coefficients of our stochastic equation are defined on a set of probability measures.
Therefore, to be able to formulate the conditions on the coefficients in the next chapters, we need to introduce spaces of measure-valued functions. \\
Let $(X,\Vert \cdot \Vert _X)$, $(Y,\Vert \cdot \Vert _Y)$ and $(E,\Vert \cdot \Vert _E)$ be Banach spaces such that
\begin{align*}
X\subset Y \subset E
\end{align*}
continuously and densely.
For $p \geq 2 $ we define the following spaces of measures and measure-valued functions:
\begin{align*}
\mathcal{P}_2(\mathcal{C}(Y)) &\cap \mathcal{P}_p(L^p _X)\\
 &:=\Big\lbrace \mu \in \mathcal{P}(L^p_E): \ \mu (\mathcal{C})=\mu (L^p_X)=1 
\text{ and } \mu \left( \Vert \cdot \Vert_{\infty}^2\right), \mu \big( \Vert \cdot \Vert_{L^p_X}^p\big) < \infty 
\Big\rbrace.
\end{align*}
Note that this set is well defined since $\mathcal{C}(Y) \subset L^p_E$ and $L^p_X \subset L^p_E$  continuous and hence by Kuratovski's theorem (\cite[Theorem 15.1]{Kechris2012} or \cite{Fake}) $\mathcal{C}(Y),L^p_X  \in \mathcal{B}\left(L^p_E\right)$. \\
Clearly
$\mathcal{P}_2(\mathcal{C}(Y)) \cap \mathcal{P}_p(L^p_X)$ is a metric space with respect to the metric $d:=\mathbb{W}_2^{\mathcal{C}(Y)}+\mathbb{W}_p^{L^p_X}$.
Define
\begin{align*}
C\left([0,\infty);\mathcal{P}_2(\mathcal{C})\right) &\cap L^p_{\text{loc}}\left([0,\infty);\mathcal{P}_p\left(L^{p}\right)\right) \\
&:= \bigg\lbrace
\mu: [0,\infty)\rightarrow\mathcal{P}_2(\mathcal{C}(Y))\cap \mathcal{P}_p(L^p_X): 
\mu:[0,\infty)
 \rightarrow \mathcal{P}_2(\mathcal{C}(Y)) 
\\ &\hspace{60pt} \text{ is continuous and } \int_{0}^t{\mu}_s(\Vert\cdot\Vert^p_{L^p})\text{d}s<\infty \ \forall t\geq 0 \bigg\rbrace.
\end{align*}
Note that if $\mu \in C\left([0,\infty);\mathcal{P}_2(\mathcal{C})\right) \cap L^p_{\text{loc}}\left([0,\infty);\mathcal{P}_p\left(L^{p}\right)\right)$,
then $\mu$ is  $\mathcal{B}([0,\infty))/ \linebreak \mathcal{B}\left(\mathcal{P}_2(\mathcal{C}(Y)) \cap \mathcal{P}_p(L^p _X)\right)$-measurable,
since $\mu:[0,\infty)
 \rightarrow \mathcal{P}_2(\mathcal{C}(Y)) \supset \mathcal{P}_2(\mathcal{C}(Y)) \cap \mathcal{P}_p(L^p _X) \text{ is continuous and}$ thereby $\mathcal{B}([0,\infty))/\mathcal{B}(\mathcal{P}_2\left(\mathcal{C}(Y))\right) \cap \mathcal{P}_p(L^p _X)$-measurable
and $\mathcal{B}(\mathcal{P}_2\left(\mathcal{C}(Y))\right) \cap \mathcal{P}_p(L^p _X)=\mathcal{B}(\mathcal{P}_2\left(\mathcal{C}(Y)) \cap \mathcal{P}_p(L^p _X)\right)$
by Kuratowski's theorem.
Define
\begin{align*}
C\left([0,T];\mathcal{P}_2(\mathcal{C})\right) &\cap L^p\left([0,T];\mathcal{P}_p\left(L^{p}\right)\right) \\
:=& \bigg\lbrace 
\mu: [0,T])\rightarrow\mathcal{P}_2(\mathcal{C}(Y))\cap \mathcal{P}_p(L^p_X):
\mu:[0,T]
 \rightarrow \mathcal{P}_2(\mathcal{C}(Y)) 
\\ &\hspace{70pt} \text{ is continuous and } \int_{0}^T{\mu}_s(\Vert\cdot\Vert^p_{L^p})\text{d}s<\infty \bigg\rbrace,
\end{align*}
where $T>0$ is fixed.
With the same argument as above, $\mu \in C\left([0,T];\mathcal{P}_2(\mathcal{C})\right)
\cap L^p\left([0,T];\mathcal{P}_p\left(L^{p}\right)\right)$
is $\mathcal{B}([0,T])/\mathcal{B}\left(\mathcal{P}_2(\mathcal{C}(Y)) \cap \mathcal{P}_p(L^p _X)\right)$-measurable.

%\newpage
\section{Distribution-Dependent SDE's with delay in finite dimensions}
The aim of this chapter is to solve the following delay-distribution dependent SDE in $\mathbb{R}^d$:
\begin{equation}
\text{d}X(t)= b(t,X_t,\mathcal{L}_{X_t})\text{d}t+ \sigma(t,X_t,\mathcal{L}_{X_t})\text{d}W(t),
\label{3}
\end{equation}
where $W=(W(t))_{t\geq 0}$ is a $d$-dimensional Brownian motion,
$ \in \mathbb{N}$,
defined on a stochastic basis $(\Omega, \mathcal{F},(\mathcal{F}_t)_{t \geq -r_0}, P)$,
with $r_0>0$ fixed
and
\begin{align*}
b &\colon  {[0,\infty)} \times \mathcal{C} \times (\mathcal{P}_2(\mathcal{C})\cap \mathcal{P}_p(L^p))
\longrightarrow \mathbb{R}^d; \\
\sigma  &\colon  {[0,\infty)} \times \mathcal{C} \times (\mathcal{P}_2(\mathcal{C})\cap \mathcal{P}_p(L^p)) \longrightarrow \mathbb{R}^{d \times d}
\end{align*}
$\mathcal{B}([0,\infty) \otimes \mathcal{B}(\mathcal{C}) \otimes \mathcal{B}(\mathcal{P}_2(\mathcal{C})\cap \mathcal{P}_p(L^p))$-measurable,
whereby $p \geq 2$ is fixed. \\

The main difficulty, compared to the well-known, classical SDE's (c.f.  \cite{Krylov} or \cite{Liu2015}),
that has to be overcome to get an existence and uniqueness result, is to deal with the delay and the distribution dependence.
To achieve such a result we first formulate certain conditions on the coefficients $b$ and $\sigma$ and define precisely what a \textit{solution} of (\ref{3}) is.
Afterwards we are going to prove existence and uniqueness of solutions to (\ref{3}).
The main inspiration for our proof comes from \cite{MP}, i.e. the existence of solutions to (\ref{3}) is derived from an existence and uniqueness result about SDDE's.
But unlike in \cite{MP}, we use the Banach fixed point theorem instead of an \textit{iteration in distribution} and use \cite[Theorem 4.2]{RZhu} instead of \cite[Corollary 4.1.2]{WangH} to the show the existence and uniqueness of SDDE's, because our conditions on the coefficients differ from those in \cite{MP}.

\subsection{Conditions on the coefficients and main result}
To show existence and uniqueness of solutions to (\ref{3}), we fix $p \geq 2$ and assume that the coefficients $b$ and $\sigma$ fulfill the following conditions. 
For simplicity we write $\mathbb{W}_2$ instead of $\mathbb{W}_2^{\mathcal{C}}$ for the Wasserstein distance on 
$\mathcal{P}_2(\mathcal{C})$ and use the notations introduced in 1.1 and 1.2.

\begin{itemize}
\item[(H1)] (Continuity)
For every $t \geq 0$, $b(t,\cdot,\cdot)$ and $\sigma(t,\cdot,\cdot)$ are continuous on $\mathcal{C} \times (\mathcal{P}_2(\mathcal{C})\cap\mathcal{P}_p(L^p))$.
\item[(H2)] (Coercivity)
There exists $\alpha \colon \mathbb{R}_+  \mapsto \mathbb{R}_+$ non-decreasing such that
\begin{align*}
\int_0^t 2\langle b(s,\xi _s, \mu _s),\xi(s) \rangle \text{d}s  
 \leq 
&- \frac{1}{2} \int_0^t \vert \xi (s) \vert^p \text{d}s
+ \alpha (t) \Vert \xi _0 \Vert _{L^p}^p 
\\ &+ \alpha (t) \int_0^t \left( 1+ \Vert \xi _s \Vert _{\infty}^2 + \mu _s (\Vert \cdot \Vert _{\infty}^2) \right) \text{d}s,
\end{align*}
for all $t \geq 0$, $\xi \in \mathcal{C}_{\infty}$ and $\mu \in  C\left([0,\infty);\mathcal{P}_2(\mathcal{C})\right)
\cap L^p_{\text{loc}}\left([0,\infty);\mathcal{P}_p\left(L^{p}\right)\right)$.
\item[(H3)] (Monotonicity)
There exists 
$\beta \colon \mathbb{R}_+  \mapsto \mathbb{R}_+$,
non-decreasing, 
such that 
\begin{align*}
&\int_0^t 2\langle b(s,\xi_s, \mu _s)-b(s,\eta _s, \nu _s),\xi(s)-\eta(s) \rangle 
 \text{d}s  
\\& \quad \leq  \beta (t) \int_0^t  \Vert \xi _s- \eta _s \Vert _{\infty}^2
+ \mathbb{W}_2(\mu _s, \nu _s)^2 \text{d}s 
+ \beta (t)   \Vert \xi _0 -  \eta _0 \Vert _{L^p}^p
\end{align*}
and
\begin{align*}
&\int_0^t \Vert \sigma(s,\xi _s, \mu _s)-\sigma(s,\eta _s, \nu _s) \Vert _{\text{HS}}^2  \text{d}s 
\\ &\quad \leq  \beta (t) \int_0^t  \Vert \xi _s- \eta _s \Vert _{\infty}^2
+ \mathbb{W}_2(\mu _s, \nu _s)^2 \text{d}s
+ \beta (t)   \Vert \xi _0 -  \eta _0 \Vert _{L^p}^p,
\end{align*}
for all $t \geq 0$; $\xi, \eta \in \mathcal{C}_{\infty}$ and  $\mu , \nu \in C\left([-r_0,\infty);\mathcal{P}_2(\mathcal{C})\right)
\cap L^p_{\text{loc}}\left([-r_0,\infty);\mathcal{P}_p\left(L^{p}\right)\right)$.
\item[(H4)] (Growth)
$b$ is bounded on bounded sets in $ {[0,\infty)} \times \mathcal{C} \times( \mathcal{P}_2(\mathcal{C})\cap \mathcal{P}_p(L^p))$,
and there exists a non-decreasing function $\gamma \colon \mathbb{R}_+ \mapsto \mathbb{R}_+$
and some $q_0 \in \mathbb{N}$ 
such that 
\begin{align*}
\int_0^t \vert b(s, \xi _s, \mu _s) \vert ^{\frac{p}{p-1}} \text{d}s 
&\leq
\gamma (t) \left( \int_{0}^t \vert \xi (s) \vert^p 
+\mu _s (\Vert \cdot \Vert_{L^p}^p)  \text{d}s 
+\Vert \xi _0  \Vert _{L^p}^p \right)^{q_0} \\
&\quad +\gamma (t) \bigg( 1+ \sup _{s \in [0,t]} \Vert \xi _s \Vert _{\infty}^{2q_0} + \sup _{s \in [0,t]} \mu _s \left(\Vert \cdot \Vert _{\infty}^2\right)^{q_0} \bigg)
\end{align*}
and
\begin{align*}
\Vert \sigma(t,\xi _t, \mu _t) \Vert ^2_{\text{HS}}
\leq \gamma (t)  \left( 1+  \Vert \xi _t \Vert _{\infty}^2 +  \mu _t (\Vert \cdot \Vert _{\infty}^2) \right),
\end{align*}
for all $t \geq 0$, $\xi \in\mathcal{C}_{\infty}$ and $\mu \in  C\left([-r_0,\infty);\mathcal{P}_2(\mathcal{C})\right)
\cap L^p_{\text{loc}}\left([-r_0,\infty);\mathcal{P}_p\left(L^{p}\right)\right)$.
\end{itemize}

Let us briefly comment on these conditions.
First of all, these conditions look similar to standard monotonicity and coercivity conditions, like they were for example formulated in \cite{Krylov} or \cite{Liu2015}.
The main difference is that, in order to deal with the delay and the distribution dependence, the sup-norm and the Wasserstein metric appear on the right hand side.
Another difference is that the conditions are in integrated form, which, as we are going to discuss in further detail in section 2.3.2 of this chapter, will be helpful for the generalization to infinite dimensions in the next chapter.\\
Moreover, in the case that $b$ and $\sigma$ are distribution independent,
i.e. $b(t,\xi,\mu)=\bar{b}(t,\xi)$ and $\sigma(t,\xi,\mu)=\bar{\sigma}(t,\xi)$,
the existence of a solution to (\ref{3}) is ensured by \cite[Theorem 4.2]{RZhu}, because,
as we will see in the proof of Lemma \ref{ExY},
for those $b$ and $\sigma$ (H1)-(H4) imply (H1)-(H5) in \cite{RZhu}.
Note that the measurability of $s \mapsto b(s, \xi _s, \mu _s)$ and $s \mapsto \sigma(s, \xi _s, \mu _s)$, with $\xi$ and $\mu$ as in the conditions is ensured by Remark \ref{SegStet} and the assumptions on $\xi$ and $\mu$.
By (H4) all integrals in (H1)-(H3) are well-defined.\\

In the following we introduce different notions of solution to (\ref{3}) and uniqueness of solutions.
\begin{defn}
A pair $(X,W)$, where $X=(X(t))_{t \geq -r_0}$ is an $(\mathcal{F}_t)$-adapted, $\mathbb{R}^d$-valued process with continuous sample paths 
and $W$ is a $\mathbb{R}^d$-valued, $(\mathcal{F}_t)$-Wiener process on a stochastic basis $(\Omega, \mathcal{F},(\mathcal{F}_t)_{t \geq -r_0},P)$
is called a  \textit{weak solution} of (\ref{3}) with initial condition $\psi \in \mathcal{C}$ iff
\begin{itemize}
\item[(i)]
\begin{equation}
\mathbb{E}[\Vert X_t \Vert _{\infty}^2] 
+ \int_{-r_0}^t\mathbb{E}[ \vert X(s)\vert^p]\text{d}s< \infty,
\label{IntBed}
\end{equation}
for all $t \geq 0$;
\item[(ii)]
\begin{equation}
X(t)=X(0)+
 \int_{0}^t  b(s,X_s,\mathcal{L}_{X_s}) \text{d}s + \int_{0}^t  \sigma(s,X_s,\mathcal{L}_{X_s}) \text{d}W(s),  
\label{DefSoleq}
\end{equation}
for all $t \geq 0$ $P$-a.s.; and
\item[(iii)]
\begin{equation}
X(t)=\psi(t),
\label{IntCondEq}
\end{equation}
for all $t \in [-r_0,0]$ $P$-a.s.
\end{itemize}
\label{DefWsol}
\end{defn}

\begin{rem}
Note that (\ref{IntBed}) implies that for every weak solution $(X(t))_{t \geq
-r_0}$ we have
\begin{align*}
\mathbb{E}\bigg[\int_0^T \Vert X_{t} \Vert _{L^p}^p \text{d}t +
\sup_{t \in [0,T]} \Vert X_{t} \Vert _{\infty}^2 \bigg] < \infty \quad \forall \, T \geq 0
\end{align*}
and $(X_t)_{t \in  [0,T]}$ is a continuous $\mathcal{C}$-valued process.
This, together with Lebesgues theorem, implies that $[0,\infty) \ni t \mapsto \mathcal{L}_{X_{t}} $
is a continuous map from $[0,\infty)$ to $(\mathcal{P}_2(\mathcal{C}),\mathbb{W}_2)$.
By Kuratowski's theorem and (\ref{IntBed}), this implies that $[0,\infty) \ni t \mapsto \mathcal{L}_{X_t}$ is $\mathcal{B}([0,\infty))/ \\ \mathcal{B}(\mathcal{P}_2(\mathcal{C})\cap \mathcal{P}_p(L^p))$-measurable.\\
In particular, $(t,\omega) \mapsto  b(t, X_{t}(\omega), \mathcal{L}_{ { X_{t}}})$
and $(t,\omega) \mapsto  \sigma(t, X_{t}(\omega), \mathcal{L}_{{ X_{t}}})$
are progressively measurable maps.
Thus the integrals on the right-hand side of (\ref{DefSoleq}) are well-defined.
\label{StetRem}
\end{rem}

\begin{defn}
We say (\ref{3}) has a  \textit{(strong) solution} if for
every stochastic basis $(\Omega, \mathcal{F},(\mathcal{F}_t)_{t \geq -r_0},P)$
with a given $\mathbb{R}^d$-valued, $(\mathcal{F}_t)$-Wiener process $W$
and given initial condition $\psi \in \mathcal{C}$,
there exists a $(\mathcal{F}_t)$-adapted, continuous $\mathbb{R}^d$-valued process $X$ such that $X$  fulfills (\ref{IntBed})-(\ref{IntCondEq}) in Definition \ref{DefWsol}.
\label{DefSol}
\end{defn}

The next definitions recall different notions of uniqueness (c.f. \cite[Appendix E]{RZhu}).
\begin{defn}
\label{DefEin1}
We say that \textit{weak uniqueness} holds for (\ref{6}) if whenever $(X,W)$ and $(\tilde{X},\tilde{W})$ are weak solutions with stochastic basis
$(\Omega, \mathcal{F},(\mathcal{F}_t)_{t \geq -r_0},P)$ and \linebreak $(\tilde{\Omega}, \tilde{\mathcal{F}},(\tilde{\mathcal{F}}_t)_{t \geq -r_0},\tilde{P})$
such that
\begin{align*}
X_0=\tilde{X}_0=\psi,
\end{align*}
for some $\psi \in \mathcal{C}$, then
\begin{align*}
P \circ X^{-1} =\tilde{P} \circ \tilde{X}^{-1}
\end{align*}
as measures on $\left(\mathcal{C}_{\infty},\mathcal{B}(\mathcal{C}_{\infty})\right)$.
\end{defn}

\begin{defn}
\label{DefEin2}
We say that \textit{path-wise uniqueness} holds for (\ref{6}), if whenever $(X,W)$ and $(\tilde{X},W)$ are two weak solutions on the same stochastic basis \linebreak
$(\Omega, \mathcal{F},(\mathcal{F}_t)_{t \geq -r_0},P)$
and with the same Wiener process $W$ on $(\Omega, \mathcal{F},(\mathcal{F}_t)_{t \geq -r_0},P)$ such that $X_0=\tilde{X}_0$ $P$-a.s., then
\begin{align*}
X(t)=\tilde{X}(t),
\end{align*}
for all $t \geq 0$ $P$-a.s.
\end{defn}
The next Theorem is the main result of this chapter and shows the existence of a unique strong solution as well as weak uniqueness.

\begin{thm}
\label{MT2}
Assume (H1)-(H4).
\begin{itemize}
\item[(a)]
For any $\psi  \in \mathcal{C}$, (\ref{3}) has a (pathwise) unique (strong) solution $(X(t))_{t \geq -r_0}$,
fulfilling $ X_{0}=\psi $.
Moreover
\begin{equation}
\label{qEnd}
\mathbb{E}\bigg[\sup _{t \in [-r_0,T]}\vert X(t)\vert^{2q} \bigg]<\infty,
\end{equation}
for all $T>0$ and $q \in \mathbb{N}$.
\item[(b)]
Whenever $(X,W)$ and $(Y,W)$ are weak solutions of (\ref{6}) on a stochastic basis $(\Omega, \mathcal{F},(\mathcal{F}_t)_{t \geq -r_0},P)$, we have
\begin{itemize}
\item[(i)]
\begin{align}
\mathbb{E}&\bigg[\sup_{t \in [-r_0,T]}\vert X(t) - Y(t) \vert ^2 \bigg] \label{EindAb} \\
&\leq \inf_{\epsilon \in (0,1)} \Bigg\lbrace \Bigg(\frac{\mathbb{E}\left[ \Vert X_0 -Y_0 \Vert_{\infty}^2\right]}{1-\epsilon} +2\beta (t)\bigg(\frac{\epsilon +6}{(1-\epsilon)\epsilon}\bigg)\mathbb{E}\big[\Vert X_0 -Y_0 \Vert_{L^p}^p\big]\Bigg) \notag
\\& \hspace{164pt} \cdot \exp\left(4\beta(t)\left(\frac{\epsilon +3}{(1-\epsilon)\epsilon} \right)t\right)\Bigg\rbrace. \notag
\end{align}
\item[(ii)]
\begin{align*}
\mathbb{E}\bigg[\sup _{r \in [-r_0,T]}\vert X(r) \vert^2 &+ \int_0^T \vert X(s) \vert ^p \text{d}s \bigg] \\
 &\leq H(T)\left(1 + \mathbb{E}\left[\Vert X_0 \Vert^2_{\infty}\right] 
+\mathbb{E}\left[\Vert X_0  \Vert_{L^p}^p\right]\right),
\end{align*}
for all $T>0$ and some non-decreasing function $H: \mathbb{R}_+ \rightarrow \mathbb{R}_+$.
\end{itemize}
\item[(c)]
(\ref{3}) has weak uniqueness.
\end{itemize}
\end{thm}

\subsection{Proof of the main result}
We are going to prove the main result by using the Banach fixed-point theorem.
Fix a stochastic basis $(\Omega,  \mathcal{F}, (\mathcal{F}_t)_{t \geq -r_0}, P)$,
a $d$-dimensional $(\mathcal{F}_t)$-Brownian motion $(W(t))_{t \geq -r_0}$ and an initial condition $\psi \in \mathcal{C}$.
For $T>0$ and $q \in  \mathbb{N}$ define
\begin{align*}
E^q(T):=\big\lbrace X \in L^{2q}(\Omega,  \mathcal{F},P; & C([-r_0,T];\mathbb{R}^d)): (X(t))_{t \in [-r_0,T]} \text{ is a }\\
&(\mathcal{F}_t)_{t \in [-r_0,T]}-\text{adapted, continuous process} \big\rbrace.
\end{align*}
Clearly, $E^q(T)$ is a Banach space with respect to the norm
\begin{align*}
\Vert X \Vert _{E^q(T)}^{2q} := \mathbb{E}\bigg[\sup _{t \in [-r_0,T]} \vert X(t) \vert^{2q} \bigg].
\end{align*}
Moreover define
\begin{align*}
E^q:=\left\lbrace X: \Omega \times [-r_0, \infty) \rightarrow \mathbb{R}^d: \
X \mid _{\Omega \times [-r_0, T]} \in E^q(T) \ \forall T>0 \right\rbrace.
\end{align*}
Next, solve for any $X \in E^q$ the classical path-dependent SDE
\begin{equation}
\label{6}
\begin{cases} 
\text{d}Y(t)= b(t,Y_{t},\mu_{t})\text{d}t+ \sigma(t,Y_{t},\mu_{t})\text{d}W(t), \ t \geq 0, \\
Y_0=\psi ,
\end{cases}
\end{equation}
where $\mu_{t}:=\mathcal{L}_{X_{t}}$.

Before we can prove Theorem \ref{MT2} we need the following two lemma.
The first lemma deals with the existence of solutions to (\ref{6}).

\begin{lem}
\label{ExY}
Assume (H1)-(H4). Then for any $X \in E^q$, $q \geq \frac{p}{2}$ and any initial condition $\psi \in \mathcal{C}$, (\ref{6}) has a unique solution $Y \in E^q$, i.e. there exists a unique continuous, adapted, $\mathbb{R}^d$-valued processes $(Y(t))_{t\geq -r_0}$ which fulfills (\ref{6}).
Moreover, for all $T>0$, 
\begin{equation*}
\mathbb{E}\bigg[\sup_{t \in [-r_0,T]} \vert Y(t) \vert ^{2q} \bigg] < \infty.
\end{equation*}
\end{lem}
\begin{proof}
Define $\bar{b}(t,\xi):=b(t,\xi ,\mu_t)$ and 
$\bar{\sigma}(t,\xi):=\sigma(t,\xi ,\mu_t)$,
$(t,\xi)\in [0,\infty)\times\mathcal{C}$.
Now it is easy to see that $\bar{b}$ and $\bar{\sigma}$ fulfill (H1)-(H5) in \cite[Theorem 4.2]{RZhu}.
\end{proof}

Now take $T>0$ and $q \in \mathbb{N}$ with $q \geq \frac{p}{2}$ fixed but arbitrary. For $X \in E^q(T)$ define $\Lambda X \in E^q(T)$ as the unique solution to (\ref{6}) up to time $T$ and $\Lambda X \in E^q(T)$.
$\Lambda :E^q(T) \rightarrow E^q(T)$ is a well-defined mapping, since we can extend every $X \in E^q(T)$ to an element $\tilde{X}\in E^q$ by setting $\tilde{X}(t):=X(T)$ for $t >T$ and apply Lemma \ref{ExY} to $\tilde{X}$
in order to get a solution up to infinity and therefore up to time $T>0$.
The path-wise uniqueness up to time $T$ can be proved as in \cite{RZhu} or as in the proof of Theorem \ref{MT2} (ii) below. \\
If $X \in E^q(T)$ is a fixed-point of $\Lambda$, i.e. $\Lambda X=X$, we have for $t \in [0,T]$ that
\begin{align*}
X(t)=\Lambda X (t)&=
\Lambda X(0)+\int_0^t b(s, (\Lambda X)_s, \mathcal{L}_{X_s})\text{d}s
+\int_0^t \sigma(s, (\Lambda X)_s, \mathcal{L}_{X_s})\text{d}W(s)\\
&= X(0)+\int_0^t b(s, X_s, \mathcal{L}_{X_s})\text{d}s
+\int_0^t \sigma(s, X_s, \mathcal{L}_{X_s})\text{d}W(s)
\ \ P-\text{a.s.}
\end{align*}
and
\begin{align*}
X(t)=\Lambda X (t)=\psi (t)
\end{align*}
for all $t \in [-r_0,0]$ $P$-a.s.
Thus $X$ is a solution of (\ref{6}) up to time $T$. 
Therefore our next step is to show that $\Lambda$ fulfills the conditions of the generalized Banach fixed-point theorem.

\begin{lem}
There exists $K_q: \mathbb{R}_+ \rightarrow \mathbb{R}_+$ non-decreasing such that for all $X,Y \in E^q(T)$ and $n \in \mathbb{N}$
\begin{equation}
\mathbb{E}\bigg[\sup _{t \in [-r_0,T]} \vert \Lambda^n X(t)- \Lambda ^n Y(t) \vert ^{2q} \bigg]
\leq  K_q(T)^n\frac{T^n}{n!}\mathbb{E}\bigg[\sup _{t \in [-r_0,T]} \vert X(t) -Y(t) \vert^{2q} \bigg].
\label{Kontra}
\end{equation}
(Whereby $\Lambda ^n X$ means, that $\Lambda$ is applied $n$-times to $X$.)
\label{KontraLem}
\end{lem}

\begin{proof}
For $n \in \mathbb{N}_0$ define $X^{(n)}:=\Lambda^n X$ and $Y^{(n)}:=\Lambda^n Y$. 
Moreover define $\mu _t^{(n)}:=\mathcal{L}_{X^{(n)}_t}$ and $\nu _t^{(n)}:=\mathcal{L}_{Y^{(n)}_t}$,  $t \in [0,T]$.
By the definition of $\Lambda$ we have for $n \geq 1$,
that $X^{(n)}=\Lambda(X^{(n-1)})$ solves
\begin{equation*}
\begin{cases} 
\text{d}X^{(n)}(t)= b(t,X^{(n)}_{t},\mu _t^{(n-1)})\text{d}t+ \sigma(t,X^{(n)}_{t},\mu _t^{(n-1)})\text{d}W(t), \ t \in [ 0,T], \\
X^{(n)}_0=\psi 
\end{cases}
\end{equation*}
and that $Y^{(n)}=\Lambda(Y^{(n-1)})$ solves
\begin{equation*}
\begin{cases} 
\text{d}Y^{(n)}(t)= b(t,Y^{(n)}_{t},\nu _t^{(n-1)})\text{d}t+ \sigma(t,Y^{(n)}_{t},\nu _t^{(n-1)})\text{d}W(t), \ t \in [ 0,T], \\
X^{(n)}_0=\psi .
\end{cases}
\end{equation*}
Applying Itô's formula to $\vert X^{(n)}(t)-Y^{(n)}(t) \vert^2$  and using (H3)
one can prove as in \cite[Lemma 3.2]{MP} that we have for $t \in [0,T]$ and $n \in \mathbb{N}$:
\begin{equation*}
\begin{split}
\mathbb{E}&\bigg[ \sup _{r \in [-r_0,t]} \vert X^{(n)}(r)-Y^{(n)}(r) \vert^{2q} \bigg]
\\ &\leq
2(C_q(T) + \tilde{C}_q(T)) \mathbb{E}\bigg[ \int_0^t \sup_{r \in [-r_0,s]} \vert  X^{(n)}(r)-Y^{(n)}(r) \vert ^{2q}\text{d}s \bigg] \\
&\qquad + 2(C_q(T) + \tilde{C}_q(T)) \mathbb{E}\bigg[\int_0^t \sup_{r \in [-r_0,s]} \vert  X^{(n-1)}(r)-Y^{(n-1)}(r) \vert ^{2q} \text{d}s \bigg],
\end{split}
\end{equation*}
with $C_q$, $\tilde{C}_q: \mathbb{R}_+ \rightarrow \mathbb{R}_+$ non-decreasing.
Now (\ref{Kontra}) follows from Gronwall's Lemma.\newline
\end{proof}

\subsubsection{Proof of Theorem \ref{MT2}}
Now we can prove Theorem \ref{MT2}.
\begin{proof}
(a):
By Lemma \ref{KontraLem} we have for all $q \in \mathbb{N}$ with $q \geq\frac{p}{2}$ and $T>0$
\begin{equation*}
\Vert \Lambda^n X - \Lambda^n Y \Vert _{E^q(T)} \leq \left(\frac{(K_q(T)T)^n}{n!}\right)^{\frac{1}{2q}} \Vert  X -  Y \Vert _{E^q(T)},
\end{equation*}
for all $n \in \mathbb{N}$ and $X,Y \in E^q(T)$.
Thus, by the generalized Banach fixed-point theorem, $\Lambda$ has a unique fixed-point $X \in E^q(T)$.
As discussed above, this means that $X$ is a a solution of (\ref{3}) up to time $T$.
Since $T>0$ was taken arbitrarily and the (pathwise) uniqueness is ensured by (b), this implies that (\ref{3}) has a unique solution up to every time $T>0$.
Hence (\ref{3}) has solution up to infinity.\\
Since $q \in \mathbb{N}$ with $q \geq \frac{p}{2}$ was taken arbitrarily, $X$, fulfills (\ref{qEnd}) for all $T>0$ and $q\in \mathbb{N}$.

(b):
(i):
Let $(X,W)$ and $(Y,W)$  be two weak solutions of (\ref{3}) defined on a stochastic basis $(\Omega,\mathcal{F},(\mathcal{F})_{t \geq -r_0},P)$.
By Itô's formula and (H3) we have for all $t\geq0$
\begin{align*}
&\vert X(t)-Y(t) \vert^2 
\\=& \vert X(0)-Y(0) \vert^2 +
\int_0^t 2 \left\langle b(s,X_s,\mu _s)-b(s,Y_s,\nu _s), X(s)-Y(s) \right \rangle  \text{d}s \\
&+ \int_0^t \Vert \sigma(s,X_s,\mu _s)-\sigma(s,Y_s,\nu _s) \Vert_{\text{HS}}^2 \text{d}s\\
&+2 \int_0^t \left\langle   X(s)-Y(s),\left\lbrace\sigma(s,X_s,\mu _s)-\sigma(s,Y_s,\nu _s)\right\rbrace \text{d}W(s) \right\rangle \\
\leq & \Vert X_0-Y_0 \Vert^2_{\infty}
+2\beta (t)\Vert X_0 -Y_0 \Vert_{L^p}^p+
2\beta (t) \int_0^t \Vert Y_s-Y_s \Vert _{\infty}^2
+ \mathbb{W}_2(\mu _s, \nu _s)^2 \text{d}s 
\\
&+2 \sup_{r \in [0,t]} \bigg\vert  \int_0^r \left\langle  X(s)-Y(s),\left\lbrace\sigma(s,X_s,\mu _s)-\sigma(s,Y_s,\nu _s)\right\rbrace  \text{d}W(s) \right\rangle \bigg\vert,
\end{align*}
where $\mu_t:= \mathcal{L}_{X_t}$ and $\nu_t:= \mathcal{L}_{Y_t}$, $t\geq0$. \\
Obviously this estimate is also true if $t \in [0,-r_0]$. 
Thus
\begin{align*}
&\sup _{r \in [-r_0,t]} \vert X(r)-Y(r) \vert^2 \\
\leq & \Vert X_0-Y_0 \Vert^2_{\infty}
+2\beta (t)\Vert X_0 -Y_0 \Vert_{L^p}^p+
2\beta (t) \int_0^{t \vee 0} \Vert Y_s-Y_s \Vert _{\infty}^2
+ \mathbb{W}_2(\mu _s, \nu _s)^2 \text{d}s \\
&+2 \sup_{r \in [0,t \vee 0]} \bigg\vert  \int_0^r \left\langle  X(s)-Y(s),\left\lbrace\sigma(s,X_s,\mu _s)-\sigma(s,Y_s,\nu _s)\right\rbrace   \text{d}W(s) \right\rangle \bigg\vert.
\end{align*}
By the BDG, Young's inequality and (H3) and we have for all $\epsilon \in (0,1)$ and $t\geq0$
\begin{align*}
&2\mathbb{E}\bigg[\sup _{r \in [0,t]} \bigg\vert  \int_0^r \left\langle  X(s)-Y(s),\left\lbrace\sigma(s,X_s,\mu _s)-\sigma(s,Y_s,\nu _s)\right\rbrace  \text{d}W(s) \right\rangle \bigg\vert\bigg] \\
\leq &6 \mathbb{E} \bigg[\bigg(\int_0^t\vert X(s)-Y(s) \vert^{2} 
\Vert \sigma(s,X_s,\mu _s)-\sigma(s,Y_s,\nu _s) \Vert_{\text{HS}}^2 \text{d}s \bigg)^{\frac{1}{2}} \bigg] \\
\leq &6 \mathbb{E} \bigg[ \sup _{s \in [0,t]} \vert X(s)-Y(s) \vert
\bigg(\int_0^t \Vert \sigma(s,X_s,\mu _s)-\sigma(s,Y_s,\nu _s) \Vert_{\text{HS}}^2\bigg)^{\frac{1}{2}} \bigg] \\
\leq &\epsilon \mathbb{E} \bigg[ \sup _{s \in [0,t]} \vert X(s)-Y(s) \vert^{2} \bigg]
+\frac{6}{\epsilon}\mathbb{E}\bigg[\bigg(\int_0^t \Vert \sigma(s,X_s,\mu _s)-\sigma(s,Y_s,\nu _s)\Vert_{\text{HS}}^2 \text{d}s\bigg)   \bigg] \\
\leq & \epsilon \mathbb{E} \bigg[ \sup _{s \in [-r_0,t]} \vert X(s)-Y(s) \vert^{2} \bigg] 
\\& \qquad +\frac{6\beta(t)}{\epsilon}\mathbb{E}\bigg[\int_0^t \Vert X_s -Y_s \Vert_{\infty}^2
+ \mathbb{W}_2(\mu _s, \nu _s)^{2} \text{d}s
+ \Vert X_0 -Y_0 \Vert_{L^p}^p     \bigg].
\end{align*}
Thus, using $\mathbb{W}_2(\mu _r, \nu _r) ^2 \leq \mathbb{E}\left[\Vert X_s - Y_s \Vert_{\infty}^2\right]\leq
\mathbb{E}\left[\sup _{r \in [-r_0,s]}\vert X(r)-Y(r) \vert ^2 \right]$,
\begin{align*}
&\mathbb{E}\bigg[ \sup_{s \in [-r_0,t]}\vert X(s)-Y(s) \vert^2 \bigg] \\
\leq& \mathbb{E}\left[ \Vert X_0 -Y_0 \Vert_{\infty}^2\right]
+2\beta (t)\left(\frac{\epsilon +3}{\epsilon}\right)\mathbb{E}\left[\Vert X_0 -Y_0 \Vert_{L^p}^p\right]
+\epsilon \mathbb{E} \bigg[ \sup _{s \in [-r_0,t]} \vert X(s)-Y(s) \vert^{2} \bigg]\\
&+4\beta(t)\left(\frac{\epsilon +3}{\epsilon} \right)\int_0^t \mathbb{E}\bigg[ \sup_{r \in [-r_0,s]}\vert X(r)-Y(r) \vert^2 \bigg]\text{d}s.
\end{align*}
Thus
\begin{align*}
\mathbb{E}\bigg[ \sup_{s \in [-r_0,t]}\vert X(s)-Y(s) \vert^2 \bigg]
\leq& \frac{\mathbb{E}\left[ \Vert X_0 -Y_0 \Vert_{\infty}^2\right]}{1-\epsilon}
+2\beta (t)\left(\frac{\epsilon +3}{(1-\epsilon)\epsilon}\right)\mathbb{E}\left[\Vert X_0 -Y_0 \Vert_{L^p}^p\right]
\\
&+4\beta(t)\bigg(\frac{\epsilon +3}{(1-\epsilon)\epsilon} \bigg)\int_0^t \mathbb{E}\bigg[ \sup_{r \in [-r_0,s]}\vert X(r)-Y(r) \vert^2 \bigg]\text{d}s.
\end{align*}
Hence, by Gronwall,
\begin{align*}
\mathbb{E}\bigg[ \sup_{s \in [-r_0,t]}\vert X(s)-Y(s) \vert^2 \bigg]
\leq &\bigg(\frac{\mathbb{E}\left[ \Vert X_0 -Y_0 \Vert_{\infty}^2\right]}{1-\epsilon}+2\beta (t)\left(\frac{\epsilon +3}{(1-\epsilon)\epsilon}\right)\\
&\hspace{1cm}\cdot \mathbb{E}\left[\Vert X_0 -Y_0 \Vert_{L^p}^p\right]\bigg) \exp\bigg(4\beta(t)\left(\frac{\epsilon +3}{(1-\epsilon)\epsilon} \right)t\bigg).
\end{align*}
Since $\epsilon \in (0,1)$ was chosen arbitrarily, \eqref{EindAb} follows. \\

\noindent
(ii):
By Itô's formula, (H2) and (H4) we have for all $ t\geq 0$
\begin{align*}
& \vert X(t) \vert^2 
\\=& \vert X(0) \vert^2 + \int_0^t 2 \left\langle b(s,X_s,\mu _s), X(s) \right \rangle + \Vert \sigma(s,X_s,\mu _s) \Vert_{\text{HS}}^2 \text{d}s 
\\ &\hspace{32pt} +2 \int_0^t \left\langle   X(s),\sigma(s,X_s,\mu _s)\text{d}W(s) \right\rangle \\
\leq&  \Vert X_0 \Vert^2_{\infty} -\frac{1}{2}\int_0^{t} \vert X(s) \vert ^p \text{d}s +
(\alpha (t)+ \gamma (t)) \int_0^t \left(1+ \Vert X_s \Vert _{\infty}^2
+ \mu _s\left(\Vert \cdot \Vert^2 _{\infty}\right) \right)\text{d}s 
\\ &\hspace{32pt} +\alpha (t)\Vert X_0\Vert_{L^p}^p +2 \sup_{r \in [0,t]} \bigg\vert  \int_0^r \left\langle  X(s),\sigma(s,X_s,\mu _s)  \text{d}W(s) \right\rangle \bigg\vert,
\end{align*}
where $\mu _t:= \mathcal{L}_{X_t}$, $t\geq 0$.
Obviously this estimate also holds true for $t \in [-r_0,0]$.
By the BDG, Young's inequality and (H4) one can prove in the same way as in (i) that
\begin{align*}
2 \mathbb{E}\bigg[\sup _{r \in [0,t]}& \bigg\vert  \int_0^r \left\langle  X(s),\sigma(s,X_s,\mu _s)   \text{d}W(s) \right\rangle \bigg\vert \bigg] \\
&\leq  \frac{1}{2} \mathbb{E}\bigg[\sup _{r \in [-r_0,t]} \vert X(r) \vert^2 \bigg]
+\gamma (t) \int_0^t\left(1+ \Vert X_s \Vert _{\infty}^2
+ \mu _s\left(\Vert \cdot \Vert^2 _{\infty}\right) \right)\text{d}s.
\end{align*}
Now fix $T>0$.
Taking expectation, the two estimates above imply for all $t \in [0,T]$
\begin{align*}
\mathbb{E}\bigg[\sup _{r \in [-r_0,t]}\vert X(r) \vert^2 \bigg]
\leq \mathbb{E}&\left[\Vert X_0 \Vert^2_{\infty}\right]
+\alpha (T)\mathbb{E}\left[\Vert X_0  \Vert_{L^p}^p\right]
\\& +
(\alpha (T)+ 2\gamma (T)) \int_0^{t } \left(1+ \mathbb{E}\left[\Vert X_s \Vert _{\infty}^2\right]
+ \mu _s\left(\Vert \cdot \Vert^2 _{\infty}\right) \right) \text{d}s\\
&-\frac{1}{2}\int_0^t \mathbb{E}\left[\vert X(s) \vert ^p \right]\text{d}s
+ \frac{1}{2}\mathbb{E}\bigg[\sup _{r \in [-r_0,t]} \vert X(r) \vert^2 \bigg].
\end{align*}
Thus, using that $\mu _s\left(\Vert \cdot \Vert^2 _{\infty}\right) =\mathbb{E}\left[\Vert X_s \Vert^2_{\infty}\right]$ for all $s \geq 0$ by the general transformation rule,
\begin{align*}
&\mathbb{E}\bigg[\sup _{r \in [-r_0,t]}\vert X(r) \vert^2 \bigg]
+\int_0^t \mathbb{E}\left[\vert X(s) \vert ^p \right]\text{d}s \\
\leq & 2\mathbb{E}\left[\Vert X_0 \Vert^2_{\infty}\right]
+2 \alpha (T)\mathbb{E}\left[\Vert X_0  \Vert_{L^p}^p\right]
+2(\alpha (T)+ 2\gamma (T)) \int_0^{t } \left(1+ 2\mathbb{E}\left[\Vert X_s \Vert _{\infty}^2\right] \right)\text{d}s \\
\leq & 2\mathbb{E}\left[\Vert X_0 \Vert^2_{\infty}\right]
+2 \alpha (T)\mathbb{E}\left[\Vert X_0  \Vert_{L^p}^p\right]
\\&\qquad +2(\alpha (T)+ 2\gamma (T)) \int_0^{t } \bigg(1+ 2\mathbb{E}\bigg[\sup _{r \in [-r_0,s]} \Vert X(r) \Vert _{\infty}^2\bigg]
+ \int_0^s  \mathbb{E}\left[\vert X(r) \vert ^p \right]\text{d}r  \bigg)\text{d}s.
\end{align*}
Thus, by the Gronwall lemma, there exists $H: \mathbb{R}_+ \rightarrow \mathbb{R}_+$ non-decreasing such that
\begin{equation*}
\mathbb{E}\bigg[\sup _{r \in [-r_0,T]}\vert X(r) \vert^2 
+\int_0^T \vert X(s) \vert ^p \text{d}s \bigg] 
\leq H(T)\left(1 + \mathbb{E}\left[\Vert X_0 \Vert^2_{\infty}\right] 
+\mathbb{E}\left[\Vert X_0  \Vert_{L^p}^p\right]\right).
\end{equation*}
Since $T$ was taken arbitrarily, this estimate holds for all $T>0$. \\

\noindent
(c): Same as in \cite[Theorem 3.1 (3)]{MP}.
\end{proof}

\begin{rem}
Let $\psi \in \mathcal{C}$ be arbitrary.
Define $X^{(0)}(t):=\psi (t \wedge 0)$, $t\geq -r_0$ and $X^{(n)}:=\Lambda X^{(n-1)}$, $n\in \mathbb{N}$, whereby $\Lambda$ is defined as before.
By the definition of $\Lambda$, $X^{(n)}$ solves
\begin{align*}
\begin{cases}
&\text{d}X^{(n)}(t)=b(t,X^{(n)}_t,\mathcal{L}_{X^{(n-1)}_t})\text{d}t
+\sigma(t,X^{(n)}_t,\mathcal{L}_{X^{(n-1)}_t})\text{d}W(t) \\
&X^{(n)}_0=\psi.
\end{cases}
\end{align*}
As we know from the Banach fixed point theorem and Lemma \ref{KontraLem},
$X^{(n)}=\Lambda^{n}X^{(0)} \rightarrow X$ in $E^q(T)$ as $n \rightarrow \infty$, for all $T>0$.
In \cite{MP} the authors prove the convergence of the $X^{(n)}$ directly and show that the limit is a solution to (\ref{3}).
Therefore the \textit{iteration in distribution}, which is used in \cite{MP}, is contained in our proof.
\end{rem}

\section{Distribution-Dependent SDEs with delay in infinite dimensions}
The goal of this chapter is to obtain a result for the existence and uniqueness of solutions of distribution-dependent SDE's with delay in infinite dimensions.
We achieve this by following the idea of \cite[Chapter 4]{Liu2015}, i.e. we approximate with solutions of finite dimensional distribution-dependent SDE's with delay (Galerkin approximation).\\

Throughout this chapter we fix a Gelfand triple $(V,H,V^*)$,
a stochastic basis \\ $(\Omega,  \mathcal{F}, ( {\mathcal{F}}_t)_{t \geq -r_0},  P)$,
 $r_0>0$, $T>0$, $p \geq 2$ and $p^* := \frac{p}{p-1}$. \\
Since
\begin{align*}
V \subset H  \subset V^*
\end{align*}
continuous and densely, we have
\begin{align*}
L^p_V \subset L^2_H  \subset L^{p^*}_{V^*}
\end{align*}
continuous and densely. 
By Kuratowski's theorem we have $L^p_V \in\mathcal{B}(L^2_H)$, $L^2_H \in \mathcal{B}(L^{p^*}_{V^*})$ and
$\mathcal{B}(L^p_V)=\mathcal{B}(L^2_H) \cap L^p_V$, $\mathcal{B}(L^2_H)=\mathcal{B}(L^{p^*}_{V^*}) \cap L^2_H$.
Hence
\begin{align*}
\mathcal{P}_p \left( L^p_V \right) \subset \mathcal{P}_2 \left( L^2_H \right)  \subset \mathcal{P}_{p^*} \big(L^{p^*}_{V^*} \big),
\end{align*}
continuously.
Therefore we can define $\mathcal{P}_2(\mathcal{C}(H)) \cap \mathcal{P}_p(L^p_V)$ and $C\left([0,T];\mathcal{P}_2(\mathcal{C}(H))\right)
\cap L^p\left([0,T];\mathcal{P}_p\left(L^{p}_V\right)\right)$ as in section 1.1.
The aim of this chapter is to solve the following path-distribution dependent SDE on H:
\begin{equation}
\text{d}X(t)= A(t,X_t,\mathcal{L}_{X_t})\text{d}t+ B(t,X_t,\mathcal{L}_{X_t})\text{d}W(t),
\label{7}
\end{equation}
with $W=(W(t))_{t\in [0,T]}$, a cylindrical $Q$-Wiener process with $Q=I$,
defined on another separable Hilbert space $(U,\langle \cdot, \cdot \rangle_U)$ and with $B$ taking values in $L_2(U,H)$, but with A taking values in the larger space $V^*$. \\
By our definition of solution (see below), X will, however, take values in $H$ again. \\

\subsection{Conditions on the coefficients and main result}
In this section, the conditions on the coefficients and the main result are presented.\\
For the rest of this chapter let $\mathbb{W}_2:=\mathbb{W}_2^{L^2_H}$. \\
Throughout the rest of this chapter, we assume that $A$ and $B$ fulfill the following conditions:
\begin{itemize}
\item[(H1)](Continuity)
\begin{align*}
&A \colon [0,T] \times \left(\mathcal{C}(H) \cap L^p_V  \right) 
\times \left( \mathcal{P}_2 \left(  \mathcal{C}(H) \right) \cap  \mathcal{P}_p \left( L^p_V \right) \right)  
\mapsto V^*; \\
&B \colon [0,T] \times \left(\mathcal{C}(H) \cap L^p_V  \right) 
\times \left( \mathcal{P}_2 \left(  \mathcal{C}(H) \right) \cap  \mathcal{P}_p \left( L^p_V \right) \right)   \mapsto L_2(U,H)
\end{align*}
are $\mathcal{B}([0,T]) \otimes \mathcal{B}\left( \mathcal{C}(H) \cap  L^p_V \right)  \otimes \mathcal{B}\left(\mathcal{P}_2 ( \mathcal{C}(H) ) \cap \mathcal{P}_p ( L^p_V) \right)$-measurable.
In addition for all $t \in [0,T]$ and $v \in V$ and $u \in U$ the maps
\[
\left( \mathcal{C}(H) \cap  L^p_V \right) 
\times \left(\mathcal{P}_2 \left( \mathcal{C}(H) \right)
\cap \mathcal{P}_p \left( L^p_V \right)   \right)   \ni (\xi,\mu) 
\mapsto {}_{V^*} \langle  A(t, \xi,\mu),v \rangle _V
\]
and
\[
\left( \mathcal{C}(H) \cap  L^p_V \right) 
\times \left(\mathcal{P}_2 \left( \mathcal{C}(H) \right)
\cap \mathcal{P}_p \left( L^p_V \right)   \right)
  \ni (\xi,\mu) 
\mapsto   B(t, \xi,\mu)u
\]
are continuous.
\item[(H2)](Coercivity)
There exists $\alpha \geq 0$  such that
\begin{align*}
\int_0^t &e^{-\lambda s} \left(  2{}_{V^*} \langle  A(s, \xi _s,\mu _s), \xi (s) \rangle _V + \Vert  B(s, \xi _s,\mu _s) \Vert _{L_2(U,H)}^2 \right) \text{d}s \\
&\leq  \alpha\int_0^t e^{-\lambda s}
\left(  1+ \Vert  \xi _s \Vert _{L^2_H}^2 + \mu _s \left(\Vert  \cdot \Vert _{L^2_H}^2 \right)\right)\text{d}s
  - \frac{1}{2} \int_0^t e^{-\lambda s}\Vert \xi (s) \Vert _V^p \text{d}s,
\end{align*}
for all $\lambda \geq 0$, $t \in [0,T]$, $\xi \in C([-r_0,T];H)\cap L^p([-r_0,T];V)  $ and $\mu \in C\left([0,T];\mathcal{P}_2(\mathcal{C}(H))\right) \cap L^p\left([0,T];\mathcal{P}_p\left(L^{p}\right)\right)$.

\item[(H3)] (Monotonicity)
There exists $ \beta  \geq 0$ such that
\begin{align*}
\int_0^t &e^{-\lambda s}  2{}_{V^*} \langle  A(s, \xi _s,\mu _s) -A(s, \eta _s,\nu _s),\xi (s)-\eta (s) \rangle _V 
\\ &\leq  \beta \int_0^t   e^{-\lambda s}
\left( \Vert  \xi _s- \eta _s \Vert _{L^2_H}^2  +  \mathbb{W}_2(\mu _s,\nu _s)^2 \right)\text{d}s
\end{align*}
and
\begin{align*}
\int_0^t &e^{-\lambda s} \Vert   B(s, \xi _s,\mu _s)- B(s, \eta _s,\nu _s) \Vert _{L_2(U,H)}^2 \text{d}s
\\ &\leq \beta  \int_0^t  e^{-\lambda s}
\left( \Vert  \xi _s- \eta _s \Vert _{L^2_H}^2  +  \mathbb{W}_2(\mu _s,\nu _s)^2 \right)\text{d}s,
\end{align*}
for all $\lambda \geq 0$, $t \in [0,T]$, $\xi, \eta \in C([-r_0,T];H)\cap L^p([-r_0,T];V)  $ and 
$\mu, \nu \in C\left([0,T];\mathcal{P}_2(\mathcal{C}(H))\right)
\cap L^p\left([0,T];\mathcal{P}_p\left(L^{p}\right)\right)$.
\item[(H4)](Growth)
For all $v \in V$ ${}_{V^*} \langle  A(\cdot, \cdot,\cdot),v \rangle _V$ is bounded on bounded sets in
$[0,T]\times\left(\mathcal{C}(H) \cap L^p_V  \right) 
\times \left( \mathcal{P}_2 \left(  \mathcal{C}(H) \right) \cap  \mathcal{P}_p \left( L^p_V \right) \right)  $.
Moreover there exist $\gamma \geq 0$ such that
\begin{align*}
\int_0^t \Vert A(s,\xi _s,\mu _s) \Vert _{V^*}^{\frac{p}{p-1}}  \text{d}s
 \leq \gamma \int_0^t \Big(  1+ \Vert \xi(s) \Vert _V ^p  +  \mu_s \big(  \Vert  \cdot \Vert _{L^p_V}^{p} \big) \Big) \text{d}s
+\gamma  \Vert  \xi _0 \Vert _{L^p_V}^{p}
\end{align*}
and
\begin{align*}
\Vert B(t,\xi _t,\mu _t) \Vert _{L_2(U,H)}^2 
\leq \gamma \Big(1+ \Vert  \xi _t \Vert _{L^2_H}^{2} + \mu _t  \big( \Vert \cdot \Vert _{L^2_H}^2 \big)  \Big) ,
\end{align*}
for all $\lambda \geq 0$, $t \in [0,T]$, $\xi \in C([-r_0,T];H)\cap L^p([-r_0,T];V) $ and
$\mu \in C\left([0,T];\mathcal{P}_2(\mathcal{C}(H))\right)\cap L^p\left([0,T];\mathcal{P}_p\left(L^{p}\right)\right)$.
\end{itemize}

Let us briefly comment on these conditions.
As we will see in Lemma \ref{Xnex3}, (H1)-(H4) were chosen in such the way, that in the case $V=H=V^*=\mathbb{R}^d$ for some $d \in \mathbb{N}$, they imply (H1)-(H4) in Chapter 2.
The factor $e^{-\lambda s}$ is necessary, because in order to prove our main result, Itô's product rule will be applied to a term of the form $e^{-\lambda t}\Vert X(t)\Vert _H^2$ so that the factor $e^{-\lambda s}$ will appear under an integral on the right hand-side, which we want to estimate with our conditions. \\
Note that the measurability of $s \mapsto A(s,\xi _s,\mu _s)$ and $s \mapsto B(s,\xi _s,\mu _s)$ is ensured by (H1) and the assumptions on $\xi$ and $\mu$.
The existence of all integrals, which appear in the conditions, is ensured by (H4). \\

Next we define precisely what a solution of (\ref{7}) is.
\begin{defn}
A continuous $H$-valued process, $( {\mathcal{F}}_t)_{t \in [-r_0,T]}$-adapted\linebreak $(X(t))_{t \in [-r_0,T]}$ is called a \textit{solution} of (\ref{7}),
if it has the following properties:
\begin{itemize}
\item[(i)]
$\mathbb{E}\left[\Vert X_t \Vert_{\mathcal{C}(H)}^2\right]< \infty$ for all $t \in [0,T]$;
\item[(ii)]
For its $\text{d}t \otimes P$-equivalence class $\widehat{X}$ we have
$\widehat{X} \in  L^p ([-r_0,T] \times \Omega,\text{d}t \otimes P ;V)$
(Whereby the $\text{d}t \otimes P$-equivalence class $\widehat{X}$ of $X$ consists of all
$\tilde{X}:[-r_0,T]\times \Omega  \rightarrow V^*$, $\mathcal{B}([0,T])\otimes \mathcal{F}/\mathcal{B}(V^*)$-measurable such that $X=\tilde{X}$ $\text{d}t \otimes P$-a.e.)
\item[(iii)]
\begin{equation}
X(t)=X(0)+ \int_{0}^t  A(s,\bar{X}_s,\mathcal{L}_{\bar{X}_s})\text{d}s
+\int_{0}^t B(s,\bar{X}_s,\mathcal{L}_{\bar{X}_s})\text{d}W(s),
\label{SGl}
\end{equation}
for every $t \in [0,T]$ $P$-a.s., where $(\bar{X}_t)_{t \in [0,T]}$ is a progressively measurable, $\mathcal{C}(H)\cap L^p_V$-valued version
(Recall that $(\bar{X}_t)_{t\in[0,T]}$ is $\text{d}t \otimes P$-version of $(X_t)_{t\in[0,T]}$, if $\tilde{X}_t(\omega)=X_t(\omega)$ for $\text{d}t \otimes P$-a.e. $(t,\omega)\in [0,T]\times \Omega$.)
of the $\mathcal{C}(H)$-valued process $({X}_t)_{t \in [0,T]}$, with the property that $(\bar{X}(t))_{t \in [0,T]}:=(\bar{X}_t(0))_{t \in [0,T]}$ is a $V$-valued, progressively measurable, $\text{d}t \otimes P$-version of $(X(t))_{t\in [0,T]}$.
\end{itemize}
\label{DS}
\end{defn}
\begin{rem}
\begin{itemize}
\item[(a)]
Just like in Remark \ref{StetRem}, Definition \ref{DS}(i) implies that
$\mathbb{E}\left[\sup _{t \in [-r_0,T]} \Vert X(t) \Vert_H^2 \right] =\mathbb{E}\left[\sup _{t \in [0,T]} \Vert X_t \Vert_{\mathcal{C}(H)}^2 \right]<\infty$
and thereby that $[0,T]\ni t \mapsto \mathcal{L}_{X_t} \in \mathcal{P}_2\left(\mathcal{C}(H)\right)$ is continuous.
Hence,
\begin{align*}
 [0,T]\ni t \mapsto \mathcal{L}_{\bar {X}_t} \in \mathcal{P}_2\left(\mathcal{C}(H)\right)\cap \mathcal{P}_p\left(L^p_V\right)
\end{align*}
is $\mathcal{B}([0,T])/\mathcal{B}\left(\mathcal{P}_2\left(\mathcal{C}(H)\right)\cap \mathcal{P}_p\left(L^p_V\right)\right)$-measurable and $\mathcal{L}_{\bar {X}_t}=\mathcal{L}_{ {X}_t}$ for all $t\in [0,T]$.
In particular $ A(\cdot,\bar{X}_{\cdot},\mathcal{L}_{\bar{X}_{\cdot}})$ and $ B(\cdot,\bar{X}_{\cdot},\mathcal{L}_{\bar{X}_{\cdot}})$ are progressively measurable.
\item[(b)]
By (H4) we have
\begin{align*}
\mathbb{E}\bigg[\int_{0}^T \Vert  A(s,\bar{X}_s,\mathcal{L}_{\bar{X}_s}) \Vert^{\frac{p}{p-1}}_{V^*}
+ \Vert B(s,\bar{X}_s,\mathcal{L}_{\bar{X}_s}) \Vert_{\text{HS}}^2 \text{d}s \bigg]<\infty.
\end{align*}
This together with (a) and (b) implies that the right-hand site of (\ref{SGl}) is well-defined.
\end{itemize}
\label{RemMb}
\end{rem}
The next theorem is the main result of this chapter.
\begin{thm}
Let $A,B$ as above satisfying (H1)-(H4) and let $\psi \in \mathcal{C}(H) \cap L^P_V$.
Then there exists a unique solution $X$ to (\ref{7}) in the sense of the definition above which satisfies the initial condition $X_0=\psi$.
Moreover 
\[
\mathbb{E}\Big[\sup_{t \in [-r_0,T]} \Vert X(t) \Vert _H ^2 \Big] < \infty.
\]
\label{MT3}
\end{thm}

\subsection{Example: A porous medium type equation}
The following example is similar to \cite[Example 4.1.11.]{RZhu}, but generalized to the delay distribution dependent case.
We consider the following equation:
\begin{equation*}
\text{d}X(t)= \Delta \psi (t,X_t,\mathcal{L}_{X_t})\text{d}t+ \text{d}W(t),
\end{equation*}
whereby $\psi : \mathbb{R}_+ \times   \left(\mathcal{C}(H) \cap L^p_V  \right) 
\times \left( \mathcal{P}_2 \left(  \mathcal{C}(H) \right) \cap  \mathcal{P}_p \left( L^p_V \right) \right) \rightarrow L^{\frac{p}{p-1}}(\Lambda)$,
$\Lambda \subset \mathbb{R}^d$, open and bounded, $p \in [2,\infty[$.
We set $V:=L^p(\Lambda)$ and $H:=\left(H^{1,2}_0(\Lambda)\right)^*$ and recall the following:
\begin{lem}
\label{DualLem}
The map
\begin{align*}
\Delta : H^{1,2}_0(\Lambda) \rightarrow \left( L^p(\Lambda) \right)^*
\end{align*}
extends to a linear isometry
\begin{align*}
\Delta :  L^{\frac{p}{p-1}}(\Lambda) \rightarrow \left( L^p(\Lambda) \right)^*
\end{align*}
and for all $u \in  L^{\frac{p}{p-1}}(\Lambda)$, $v \in  L^p(\Lambda)$
\begin{equation}
\label{Dual}
{}_{V^*} \langle - \Delta u, v \rangle {}_V
= {}_{L^{\frac{p}{p-1}}} \langle u, v \rangle {}_{L^p}
= \int u(x)v(x)\text{d}x.
\end{equation}
\end{lem}
Assume that $\psi $ fulfills the following properties:
\begin{itemize}
\item[($\Psi$1)]
$\psi$ is $\mathcal{B}([0,T]) \otimes \mathcal{B}\left( \mathcal{C}(H) \cap  L^p_V \right)  \otimes \mathcal{B}\left(\mathcal{P}_2 ( \mathcal{C}(H) ) \cap \mathcal{P}_p ( L^p_V) \right)$-measurable and
for all $t \in [0,T]$ and $v \in V=L^p(\Lambda)$
the map
\[
\left( \mathcal{C}(H) \cap  L^p_V \right) 
\times \left(\mathcal{P}_2 \left( \mathcal{C}(H) \right)
\cap \mathcal{P}_p \left( L^p_V \right)   \right)   \ni (\xi,\mu) 
\mapsto \int \psi (t,\xi,\mu)(x)v(x)\text{d}x
\]
is continuous.
\item[($\Psi$2)]
There exists $\alpha \geq 0$  such that
\begin{align*}
&\int_0^t e^{-\lambda s}
\left(  2\int \psi (s, \xi _s, \mu _s)(x) \xi (s,x) \text{d}x \right) \text{d}s \\
\geq & -\alpha\int_0^t e^{-\lambda s}
\left(  1+ \Vert  \xi _s \Vert _{L^2_H}^2 + \mu _s \left(\Vert  \cdot \Vert _{L^2_H}^2 \right)\right)\text{d}s
  + \frac{1}{2} \int_0^t e^{-\lambda s}\Vert \xi (s) \Vert _V^p \text{d}s,
\end{align*}
for all $\lambda \geq 0$, $t \in [0,T]$, $\xi \in C([-r_0,T];H)\cap L^p([-r_0,T];V)  $ and $\mu \in C\left([0,T];\mathcal{P}_2(\mathcal{C}(H))\right)
\cap L^p\left([0,T];\mathcal{P}_p\left(L^{p}\right)\right)$.
\item[($\Psi$3)]
There exists $ \beta  \geq 0$ such that
\begin{align*}
&\int_0^t e^{-\lambda s} 
2\Big(  \int \left(\psi (s, \xi _s, \mu _s)(x)-\psi (s, \eta _s, \nu _s)(x)\right) \left(\xi (s,x) - \eta(s,x)\right) \text{d}x\Big)\text{d}s
\geq  0
\end{align*}
for all $\lambda \geq 0$, $t \in [0,T]$, $\xi, \eta \in C([-r_0,T];H)\cap L^p([-r_0,T];V)  $ and 
$\mu, \nu \in C\left([0,T];\mathcal{P}_2(\mathcal{C}(H))\right)
\cap L^p\left([0,T];\mathcal{P}_p\left(L^{p}\right)\right)$.
\item[($\Psi$4)]
For all $v \in V$, $\int \psi(\cdot, \cdot,\cdot)(x) v(x) \text{d}x $ is bounded on bounded sets in
$[0,T]\times\left(\mathcal{C}(H) \cap L^p_V  \right) 
\times \left( \mathcal{P}_2 \left(  \mathcal{C}(H) \right) \cap  \mathcal{P}_p \left( L^p_V \right) \right)  $.
Moreover there exist $\gamma \geq 0$ such that
\begin{align*}
\int_0^t &\Vert \psi(s,\xi _s,\mu _s) \Vert _{L^{\frac{p}{p-1}}(\Lambda)}^{\frac{p}{p-1}}  \text{d}s
 \\ &\leq \gamma \int_0^t \left(  1+ \Vert \xi(s) \Vert _V ^p  +  \mu_s \left(  \Vert  \cdot \Vert _{L^p_V}^{p} \right) \right) \text{d}s
+\gamma  \Vert  \xi _0 \Vert _{L^p_V}^{p}.
\end{align*}

\end{itemize}
Now define
$A \colon [0,T] \times \left(\mathcal{C}(H) \cap L^p_V  \right) 
\times \left( \mathcal{P}_2 \left(  \mathcal{C}(H) \right) \cap  \mathcal{P}_p \left( L^p_V \right) \right)  
\rightarrow V^*=(L^p(\Lambda))^*$ by
$$A(t,\xi,\mu):= \Delta \psi (t, \xi, \mu), \ 
(t,\xi,\mu) \in[0,T] \times \left(\mathcal{C}(H) \cap L^p_V  \right) 
\times \left( \mathcal{P}_2 \left(  \mathcal{C}(H) \right) \cap  \mathcal{P}_p \left( L^p_V \right) \right). $$
By Lemma \ref{DualLem} $A$ is well-defined and really takes values in $V^*$.
By \eqref{Dual} we have for $(t,\xi,\mu) \in [0,T]\times\left(\mathcal{C}(H) \cap L^p_V  \right) 
\times \left( \mathcal{P}_2 \left(  \mathcal{C}(H) \right) \cap  \mathcal{P}_p \left( L^p_V \right) \right)  $
and $v \in V$
\begin{align*}
{}_{V^*} \langle A(t,\xi,\mu), v \rangle _V
= - \int \psi (t,\xi,\mu)(x) v(x) \text{d}x.
\end{align*}
Now it is easy to see that ($\Psi$1)-($\Psi$4) imply (H1)-(H4).
\subsection{Proof of the main result}

Let $\lbrace e_i \mid i \in \mathbb{N}\rbrace  \subset V$ be an orthonormal basis of $H$ such that $\lin \lbrace e_i \mid i \in \mathbb{N}\rbrace$ is dense in $V$.
Define $H_n:=\lin \lbrace e_1 \cdots e_n \rbrace \subset V$
and $\Vert\cdot\Vert_{H_n}:=\Vert\cdot\Vert_{H}$.
Since $H_n$ is a finite dimensional vector space, $\Vert\cdot\Vert_{V}\leq \Vert\cdot\Vert_{V}$, $\Vert\cdot\Vert_{V}\leq \Vert\cdot\Vert_{H}$ and $\Vert\cdot\Vert_{V}\leq \Vert\cdot\Vert_{V^*}$ are equivalent on $H_n$.
Let $P_n \colon V^* \mapsto H_n$ be defined as
\begin{equation*}
P_n y:= \sum _{i=1}^n {}_{V^*} \langle y, e_i \rangle _V e_i, \ y \in V^*.
\end{equation*}
Since ${}_{V^*} \langle y, e_i \rangle _V = \langle y, e_i \rangle _H$ for $y \in H$,
the restriction of $P_n$ to $H$,
denoted by $P_n \vert _H$,
is just the orthogonal projection onto $H_n$ in $H$.
Moreover the following Lemma holds true.
\begin{lem}
Let $P_n$ be as above. Then:
\begin{itemize}
\item[(i)]
${}_{V^*} \langle z, P_n y \rangle _V={}_{V^*} \langle y, P_n z \rangle _V$ for all $y,z \in V^*$,
\item[(ii)]
${}_{V^*} \langle  P_n y,v \rangle _V={}_{V^*} \langle y, P_n v \rangle _V$
for all $y \in V^*$, $v \in V$.
\end{itemize}
\label{PLem}
\end{lem}
Let $\lbrace g_i \mid i \in \mathbb{N}\rbrace  $ be an orthonormal basis of $U$ and set
\begin{equation*}
W^{(n)}(t):=\sum _{i=1}^n \langle W(t),g_i \rangle _U g_i.
\end{equation*}
Here we define for $g \in U$
\begin{align*}
\langle W(t),g \rangle _U:=\int_{0}^t  \langle g , \cdot \rangle _U \text{d}W(s), \ t \in(0,T],
\end{align*}
where the stochastic integral is well-defined, since the map $u \mapsto \langle g , u \rangle _U$, $u \in U$, is in $L_2(U,\mathbb{R})$.
By the definition of a $Q$-Wiener process \cite[Chapter 2.5]{Liu2015}, it is clear that $(W^{(n)}(t))_{t\in [0,T]}$ is a $n$-dimensional Brownian motion on $H_n$.
In addition define $U_n:=\lin \lbrace g_1, \cdots,g_n\rbrace$ and let $\tilde{P}_n$ is the orthogonal projection onto $U_n$ in $U$.\\
Now we consider for each $n \in \mathbb{N}$ the following stochastic equation on $H_n$:
\begin{equation}
\label{9}
\begin{cases} 
\text{d}X^{(n)}(t)= P_n A(t,X^{(n)}_t,\mathcal{L}_{X^{(n)}_t})\text{d}t+ P_n B(t,X^{(n)}_t,\mathcal{L}_{X^{(n)}_t})\text{d}W^{(n)}(t), \ t \in [0,T] \\
X^{(n)}_0= P_n \psi.
\end{cases}
\end{equation}
\begin{lem}
Assume (H1)-(H4).
Then, for every $n \in \mathbb{N}$, there exists a continuous, adapted, $H_n$-valued process $X^{(n)}$ which solves (\ref{9}).
\label{Xnex3}
\end{lem}
\begin{proof}
Fix $n \in \mathbb{N}$.
It is obvious that Theorem \ref{MT2} still holds true, if we replace $\mathbb{R}^d$ with and arbitrary, finite dimensional vector space. 
Therefore we have to show, that $b(t,\xi,\mu):=P_n A(t,\xi,\mu)$ and 
$\sigma(t,\xi,\mu):=P_n B(t,\xi,\mu)\tilde{P}_n$, 
$(t,\xi,\mu) \in [0,T]\times\mathcal{C}(H_n)\times\left(\mathcal{P}_2(\mathcal{C}(H_n))\cap\mathcal{P}_p(L^p_{H_n})\right)
\subset [0,T] \times \left(\mathcal{C}(H) \cap L^p_V  \right)\times \left( \mathcal{P}_2 \left(  \mathcal{C}(H) \right) \cap  \mathcal{P}_p \left( L^p_V \right) \right)$, fulfill (H1)-(H4) in Theorem \ref{MT2}.
\end{proof}
The following lemma is crucial for the construction of a solution to (\ref{7}).
But first we fix the following notations:
Let
\begin{align*}
J&:=L^2([0,T] \times \Omega, \text{d}t \otimes P; L_2(U,H)), \\
K&:=L^p([0,T] \times \Omega, \text{d}t \otimes P; V), \\
K^*&:=(L^p([0,T] \times \Omega, \text{d}t \otimes P; V))^* \cong L^{\frac{p}{p-1}}([-r_0,T] \otimes \Omega, \text{d}t \times P; V^*).
\end{align*}
\begin{lem}
Under the assumptions of the main theorem, there exists $C \in ]0, \infty [$ such that
\begin{align}
\begin{split}
\Vert X^{(n)} \Vert _{K} 
+ \Vert A(\cdot,X^{(n)}_{\cdot},\mathcal{L}_{X^{(n)}_{\cdot}}) \Vert _{K^*} 
&+ \Vert B(\cdot,X^{(n)}_{\cdot},\mathcal{L}_{X^{(n)}_{\cdot}}) \Vert _{J} 
\\ &+ \sup _{t \in [-r_0,T]} \mathbb{E}[\Vert X^{(n)}(t) \Vert ^2 _H ]
\leq C
\end{split}
\label{10}
\end{align}
for all $n \in \mathbb{N}$.
\label{UGL}
\end{lem}
\begin{proof}
Since $P_n$ is a the orthonorgal projection of $H$ onto $H_n$ it is well known that $\Vert P_n \Vert_{L(H)}\leq 1$.
Hence for $t\in[-r_0,0]$ it is 
$\mathbb{E}[\Vert X^{(n)}(t) \Vert_H^2]=\Vert P_n \psi(t) \Vert_H^2 \leq \Vert  \psi \Vert_{\mathcal{C}(H)}^2$.
Therefore we only have to show
\begin{align*}
\Vert X^{(n)} \Vert _{K_0}
+ \Vert A(\cdot,X^{(n)}_{\cdot},\mathcal{L}_{X^{(n)}_{\cdot}}) \Vert _{K^*} 
&+ \Vert B(\cdot,X^{(n)}_{\cdot},\mathcal{L}_{X^{(n)}_{\cdot}}) \Vert _{J} 
\\ &+ \sup _{t \in [0,T]} \mathbb{E}[\Vert X^{(n)}(t) \Vert ^2 _H ]
\leq C
\end{align*}
for some $C \in ]0, \infty [$. \\
By (H4) it is even enough to show that
\begin{align*}
\Vert X^{(n)} \Vert _{K}
+ \sup _{t \in [0,T]} \mathbb{E}[\Vert X^{(n)}(t) \Vert ^2 _H ]
\leq \tilde{C}
\end{align*}
for some $\tilde{C} \in ]0, \infty [$.
By the finite-dimensional Itô formula and Lemma \ref{PLem} we have $P$-a.s.
\begin{align*}
&\Vert X^{(n)}(t) \Vert ^2 _H
\\ &= \Vert X^{(n)}(0) \Vert ^2 _H 
+\int_{0}^t \Big(2 {}_{V^*} \big \langle P_n A \big(s,X^{(n)}_{s},\mathcal{L}_{X^{(n)}_{s}}\big), X^{(n)}(s) \big \rangle {}_V 
+  \Vert Z^{(n)}(s) \Vert ^2 _{L_2(U_n,H)} \Big) \text{d}s\\
&\qquad + M^{(n)}(t) \\
&=  \Vert X^{(n)}(0) \Vert ^2 _H  +\int_{0}^t \Big(2 {}_{V^*} \big \langle   A \big(s,X^{(n)}_{s},\mathcal{L}_{X^{(n)}_{s}}\big), X^{(n)}(s) \big \rangle {}_V 
+ \Vert Z^{(n)}(s) \Vert ^2 _{L_2(U_n,H)} \Big) \text{d}s \\
&\qquad + M^{(n)}(t),
\end{align*}
for all $t \in [0,T]$, where $Z^{(n)}(s)= P_n B(s,X^{(n)}_{s},\mathcal{L}_{X^{(n)}_{s}})$, 
$U_n=\lin \lbrace g_1, \cdots , g_n \rbrace$ and
\begin{align*}
M^{(n)}(t):=2 \int_{0}^t  \left \langle X^{(n)}(s),P_n B(s,X^{(n)}_{s},\mathcal{L}_{X^{(n)}_{s}}) \text{d}W^{(n)}(s) \right \rangle _H
\ \ t \in [0,T],
\end{align*}
is a local martingale. \\
Let $(\tau _l)_{l \in \mathbb{N}}$ be $(\mathcal{F}_t)$-stopping times such that
$\Vert X^{(n)}(t \wedge \tau _l )(\omega) \Vert ^2 _V$ is uniformly bounded in
$(t,\omega) \in [0,T] \times \Omega$, $M^{(n)}((t \wedge \tau _l)$, $t \in [0,T]$, 
is a martingale for each $l \in \mathbb{N}$ and $\tau _l \uparrow T$ as $l \rightarrow \infty$.
Then for all $l \in \mathbb{N}$,  $t \in [0,T]$
\begin{align*}
&\mathbb{E}\left[\Vert X^{(n)}(t \wedge \tau _l ) \Vert ^2 _H \right] \\
=&\mathbb{E}\left[\Vert X^{(n)}(0) \Vert ^2 _H \right]
\\&\quad +\int_{0}^t \mathbb{E}\left[ 1_{[0, \tau _l]}(s) \left( 2 {}_{V^*} \big \langle  A(s,X^{(n)}_{s},\mathcal{L}_{X^{(n)}_{s}}, X^{(n)}(s) \big \rangle {}_V \right. 
+  \left. \Vert Z^{(n)}(s) \Vert ^2 _{L_2(U_n,H)} \right) \right] \text{d}s.
\end{align*}
Using the product rule, (H3) and Fubini we obtain for $\lambda \geq 0$
\begin{align*}
&\mathbb{E}\left[e^{-\lambda t} \Vert X^{(n)}(t \wedge \tau _l ) \Vert ^2 _H \right]\\
=& \mathbb{E}\left[\Vert X^{(n)}(0) \Vert ^2 _H \right]
- \int_{0}^t \lambda \mathbb{E}\left[\Vert X^{(n)}(s \wedge \tau _l ) \Vert ^2 _H \right] e^{-\lambda s} \text{d}s \\
&+  \mathbb{E}\bigg[ \int_{0}^{t \wedge \tau _l} e^{-\lambda s} \Big( 2 {}_{V^*} \big \langle  A(s,X^{(n)}_{s},\mathcal{L}_{X^{(n)}_{s}}), X^{(n)}(s) \big \rangle {}_V 
\\&\hspace{170pt}+   \Vert B(s,X^{(n)}_{s},\mathcal{L}_{X^{(n)}_{s}}) \Vert ^2 _{L_2(U_n,H)} \Big) \text{d}s  \bigg]  \\
\leq & \mathbb{E}\left[\Vert X^{(n)}(0) \Vert ^2 _H \right]
- \int_{0}^t \lambda \mathbb{E}\left[\Vert X^{(n)}(s \wedge \tau _l ) \Vert ^2 _H
\right] e^{-\lambda s} \text{d}s \\
&+ \mathbb{E}\bigg[  \int_{0}^{t \wedge \tau _l} e^{-\lambda s} \left(\alpha\left( 1+ \Vert  X^{(n)}_s \Vert _{L^2_H}^2
+ \mu _s^{(n)} (\Vert  \cdot \Vert _{L^2_H}^2)\right)
-\frac{1}{2} \Vert  X^{(n)}(s) \Vert _V^p \right) \text{d}s  \bigg],
\end{align*}
where $\mu _s^{(n)} := \mathcal{L}_{X^{(n)}_{s}}$.
Rearranging the terms yields
\begin{align*}
&\mathbb{E}\left[e^{-\lambda t} \Vert X^{(n)}(t \wedge \tau _l ) \Vert ^2 _H \right]
+ \int_{0}^t \lambda e^{-\lambda s}  \mathbb{E}\left[
\Vert X^{(n)}(s \wedge \tau _l ) \Vert ^2 _H
\right] \text{d}s
\\ &\hspace{2cm} +\frac{1}{2} \int_{0}^t  e^{-\lambda s}
\mathbb{E}\left[1_{[0,\tau _l]}(s) \Vert  X^{(n)}(s \wedge \tau _l ) \Vert _V^p \right] \text{d}s \\
\leq &  \Vert \psi \Vert ^2 _{\mathcal{C}(H)} 
+\int_{0}^t\mathbb{E}\left[ 2 \alpha e^{-\lambda s}\Vert  X^{(n)}_s \Vert _{L^2_H}^2 \right]\text{d}s 
+\alpha \int_{0}^t  e^{-\lambda s} 
\text{d}s \\
\leq & K_1  \Vert \psi \Vert ^2 _{\mathcal{C}(H)} 
+ K_2 \int_{0}^t  e^{-\lambda s} \mathbb{E}\left[
\Vert X^{(n)}(s) \Vert ^2 _H \right] \text{d}s
+ \alpha \int_{0}^t  e^{-\lambda s} 
\text{d}s,
\end{align*}
where $K_1,K_2 \geq 0$ are constants independent of $n$ and $K_2$ is independent of $\lambda$.
To obtain the first estimate it is used that by the definition of $\mu _s^{(n)}$ and $\mu _s^{(n)} (\Vert  \cdot \Vert _{L^2_H}^2)$ 
it is $\mu _s^{(n)} (\Vert  \cdot \Vert _{L^2_H}^2)=\mathbb{E}\left[\Vert X_s\Vert _{L^2_H}^2\right]$.
For the second we used Lemma \ref{A1} in the case $Y=0$, $B=H$ and $p=2$.
In addition it is used that
\begin{align*}
\Vert  X^{(n)}_0(\theta) \Vert _{H}^2 = \Vert  P_n \psi (\theta) \Vert _{H}^2 
\leq \Vert  \psi(\theta) \Vert _{H}^2
\leq \Vert  \psi\Vert _{ \mathcal{C}(H)}^2,
\end{align*}
because $P_n \vert _H$ is  the orthogonal projection onto $H_n$ in $H$
and therefore $\Vert P_n \Vert _{L(H)} \leq 1$. \\
Choosing $\lambda = K_2$ (which is possible, because $K_2$ is independent of $\lambda$), taking $l \rightarrow \infty$ and applying Fatous's lemma we get
\begin{align*}
\mathbb{E}\Big[e^{-K_2 t} \Vert X^{(n)}(t  ) \Vert ^2 _H \Big] +\frac{1}{2} \int_{0}^t  e^{-K_2 s} \mathbb{E}\Big[ \Vert  X^{(n)}(s ) \Vert _V^p \Big] \text{d}s
\leq & K_1  \Vert \psi \Vert ^2 _{\mathcal{C}(H)} + \alpha \int_{0}^t  e^{-K_2 s} 
\text{d}s,
\end{align*}
for all $t \in [0,T]$.
Here we used that by Chapter 3 the subtracted term is finite.
Now the assertion follows for the first and fourth summand in (\ref{10}).
\end{proof}

In the proof above we used the following result which is also important for the proof of the main theorem below:
\begin{lem}
Let $B$ be a Banach space, $p \geq 2$ and
$X$, $Y \in  L^p ([-r_0,T] \times \Omega,\text{d}t \otimes P ;B)$ and $\lambda \geq 0$
Then
\begin{align*}
\mathbb{E}  \bigg[ \int_0^t e^{-\lambda s} &\Vert X_s - Y_s \Vert _{L_B^p}^p \bigg]
\\&\leq r_0 \mathbb{E}  \bigg[ \int_0^t e^{-\lambda s} \Vert X(s) - Y(s) \Vert _B^p \bigg]
+ r_0 e^{\lambda r_0}  \mathbb{E}  \left[\Vert X_0 - Y_0 \Vert _{L_B^p}^p \right]
\end{align*}
for all $t \in [0,T]$.
\label{A1}
\end{lem}
\begin{proof}
Use Fubini and the transformation.

\end{proof}

\subsubsection{Proof of Theorem \ref{MT3}}
Now we can finally prove Theorem \ref{MT3}.
\begin{proof}
By Lemma \ref{UGL} and the reflexivity of the spaces $K$, $K^*$, $J$ and $L^2 ([0,T] \times \Omega,\text{d}t \otimes P ;H)$
there exist $\tilde{X} \in K$,
$Y \in K^*$, $Z \in J$ and a subsequence $n_k \overset{k\rightarrow \infty} \longrightarrow \infty$ such that:
\begin{itemize}
\item[(i)]
$X^{(n_k)} \overset{k\rightarrow \infty} \longrightarrow \tilde{X}$
weakly in $K$ and weakly in $L^2 ([0,T] \times \Omega,\text{d}t \otimes P ;H)$.
\item[(ii)]
$Y^{(n_k)}:= A(\cdot,X^{(n_k)}_{\cdot},\mathcal{L}_{X^{(n_k)}_{\cdot}}) \overset{k\rightarrow \infty} \longrightarrow Y$
weakly in $K^*$.
\item[(iii)]
$Z^{(n_k)}:= B(\cdot,X^{(n_k)}_{\cdot},\mathcal{L}_{X^{(n_k)}_{\cdot}}) \overset{k\rightarrow \infty} \longrightarrow Z$ 
weakly in $J$.
\end{itemize}
Note that $\tilde{X}$, $Y$ and $Z$ are progressively measurable, because the approximants are progressively measurable. \\
(iii), $P_{n_k} \overset{k \rightarrow \infty}{\longrightarrow } I_H$ and $\tilde{P}_{n_k} \overset{k \rightarrow \infty}{\longrightarrow } I_U$
implies that $ P_{n_k} B(\cdot,X^{(n_k)}_{\cdot},\mathcal{L}_{X^{(n_k)}_{\cdot}}) \tilde{P}_{n_k} \overset{k\rightarrow \infty} \longrightarrow Z$ 
weakly in $J$.
Therefore, since 
\begin{align*}
\int_{0}^{\cdot}  P_{n_k} B(s,X^{(n_k)}_{s},\mathcal{L}_{X^{(n_k)}_{s}})\text{d}W^{(n_k)}(s)
= \int_{0}^{\cdot}  P_{n_k} B(s,X^{(n_k)}_{s},\mathcal{L}_{X^{(n_k)}_{s}})\tilde{P}_{n_k} \text{d}W(s)
\end{align*}
and since a bounded linear operator between two Banach spaces is weakly continuous, we obtain:

\pagebreak
\begin{itemize}
\item[(iv)]
\begin{align*}
\int_{0}^{\cdot}  P_{n_k} B(s,X^{(n_k)}_{s},\mathcal{L}_{X^{(n_k)}_{s}})\text{d}W^{(n_k)}(s)
\overset{k\rightarrow \infty} \longrightarrow
\int_{0}^{\cdot} Z(s) \text{d}W(s)
\end{align*}
weakly in $\mathcal{M}^2_T(H)$, which denotes the space of continuous, square integrable martingales $M:[0,T]\times \Omega \rightarrow H$ and is equipped with the norm  $\Vert M\Vert^2_{\mathcal{M}^2_T(H)}:=\mathbb{E}\left[\sup _{t\in [0,T]} \Vert M(t) \Vert _H^2\right]$.
\end{itemize}
Now let $v \in \bigcup _{n \geq 1}H_n$ ($\subset V$) and $\varphi \in L^{\infty}([0,T]\times \Omega; \mathbb{R})$.
Using (i)-(iv), the definition of $X^{(n_k)}$, Fubini and Lemma \ref{PLem} we obtain
\begin{align*}
&\mathbb{E} \bigg[\int_{0}^{T} {}_{V^*} \langle \tilde{X}(t), \varphi (t) v \rangle _V \text{d}t \bigg] \\
=& \lim _{k \rightarrow \infty} \mathbb{E} \bigg[\int_{0}^{T} {}_{V^*} \langle X^{(n_k)}(t), \varphi (t) v \rangle _V \text{d}t \bigg] \\
=&\lim _{k \rightarrow \infty}  \mathbb{E} \bigg[\int_{0}^{T} \bigg(  {}_{V^*} \langle P_{n_k} \psi (0) , \varphi (t) v \rangle
+ {}_{V^*} \Big \langle \int_{0}^{t} P_{n_k} Y^{(n_k)}(s) \text{d}s  , \varphi (t) v \Big \rangle {}_V 
 \\& \hspace{146pt} +\Big \langle \int_{0}^{t} P_{n_k} Z^{(n_k)}(s) \text{d}W^{(n_k)}(s),  \varphi (t) v \Big \rangle _H \text{d}t
\bigg)\bigg] \\
=& \lim _{k \rightarrow \infty} \bigg( \mathbb{E} \bigg[
\langle P_{n_k} \psi (0), v \rangle _H  \int_{0}^{T} \varphi (t) v \text{d}t \bigg]
+\mathbb{E} \bigg[\int_{0}^{T} {}_{V*}
\Big\langle Y^{(n_k)}(s), 
\int_{s}^{T}  \varphi (t) v \text{d}t \Big\rangle {}_V 
\text{d}s \bigg]\\
&\hspace{35pt} +\mathbb{E} \bigg[\int_{0}^{T} \Big \langle \int_{0}^{t} P_{n_k} Z^{(n_k)}(s) \text{d}W^{(n_k)}(s),  \varphi (t) v \Big \rangle _H \text{d}t
\bigg] \bigg) \\
=& \mathbb{E} \bigg[ \int_{0}^{T} {}_{V*} \Big\langle \psi (0) 
+\int_{0}^{t} Y(s) \text{d}s + \int_{0}^{t} Z(s) \text{d}W(s),
\varphi (t) v \Big\rangle {}_V \text{d}t \bigg].
\end{align*}
Defining 
\begin{equation*}
X(t):=
\begin{cases}
\psi (0) +  \int_{0}^t Y(s) \text{d}s + \int_{0}^t Z(s) \text{d}W(s),
\  &t \in [0,T] \\
\psi(t), & t \in [-r_0,0].
\end{cases}
\end{equation*}
we have for all $v \in \bigcup _{n \geq 1} H_n$ ($\subset V$) and $\varphi \in L^{\infty}([0,T]\times \Omega; \mathbb{R})$
\begin{equation*}
\mathbb{E} \bigg[\int_{0}^{T} {}_{V^*} \langle \tilde{X}(t), \varphi (t) v \rangle \text{d}t \bigg]
=\mathbb{E} \bigg[\int_{0}^{T} {}_{V^*} \langle X(t), \varphi (t) v \rangle \text{d}t \bigg].
\end{equation*}
Thus, using that $\bigcup _{n \geq 1} H_n$ is dense in $V$ by the choice of $(e_i)_{i\in \mathbb{N}}$,
we have $X(t,\omega) = \tilde{X}(t,\omega)$ $\text{d}t \otimes P$-a.e. $(t,\omega) \in [0,T]\times\Omega$.\\

\pagebreak
\noindent
This together with $X_0=\psi \in \mathcal{C}(H)\cap L^p_V$ implies,
that for the $\text{d}t\otimes P$ equivalence class $\hat{X}$ of $X$, we have 
$\hat{X} \in L^p([-r_0,T]\times\Omega;\text{d}t\otimes P;V)$. 
\cite[Theorem 4.2.5]{Liu2015} now implies that $X$ is a continuous $H$-valued $(\mathcal{F}_t)$-adapted process,
\begin{equation*}
\mathbb{E}\bigg[\sup_{t \in [-r_0,T]} \Vert X(t) \Vert _H ^2 \bigg] < \infty.
\end{equation*}
Therefore, it remains to verify that
\begin{equation}
\label{A}
A(\cdot,\bar{X}_{\cdot}, \mathcal{L}_{\bar{X}_{\cdot}})=Y, 
B(\cdot,\bar{X}_{\cdot}, \mathcal{L}_{\bar{X}_{\cdot}})=Z,  \text{d}t \otimes P-a.e.,
\end{equation}
where $(\bar{X}_t)_{t\in [0,T]}$ is a progressively measurable, $\mathcal{C}(H)\cap L^p_V$-valued version of $X$, the existence of which can be proved as in \cite[Excercise 4.2.3., Part 2]{Liu2015}.
To prove (\ref{A}) we first take $\rho \in L^{\infty}([0,T], \text{d}t; \mathbb{R})$, non-negative.
Then (i) and Cauchy-Schwartz implies that
\begin{align*}
\mathbb{E}&\bigg[ \int_{0}^T \rho(t) \Vert \tilde{X}(t) \Vert_H^2  \text{d}t \bigg] 
=\lim _{k \rightarrow \infty}\mathbb{E} \bigg[ \int_{0}^T  \langle \rho(t) \tilde{X}(t), X^{(n_k)}(t) \rangle _H \text{d}t \bigg]\\
 &\leq
\bigg(\mathbb{E}\bigg[ \int_{0}^T \rho(t) \Vert \tilde{X}(t) \Vert_H^2  \text{d}t \bigg]\bigg)^{\frac{1}{2}} 
\liminf _{k \rightarrow \infty} \bigg( \mathbb{E} \bigg[ \int_{0}^T \rho(t) \Vert X^{(n_k)}(t) \Vert_H^2  \text{d}t  \bigg]\bigg)^{\frac{1}{2}}.
\end{align*}
Since $X = \tilde{X}=\bar{X}$ $\text{d}t \otimes P$-a.e. on $[0,T]\times \Omega$ and $\Big(\mathbb{E}\Big[ \int_{0}^T \rho(t) \Vert \bar{X}(t) \Vert_H^2  \text{d}t \Big]\Big)^{\frac{1}{2}} < \infty$,
this implies
\begin{equation}
\label{FAb}
\mathbb{E}\bigg[ \int_{0}^T \rho(t) \Vert \bar{X}(t) \Vert_H^2  \text{d}t \bigg]^{\frac{1}{2}} 
\leq \liminf _{k \rightarrow \infty} \bigg( \mathbb{E} \bigg[ \int_{0}^T \rho(t) \Vert X^{(n_k)}(t) \Vert_H^2  \text{d}t  \bigg]\bigg)^{\frac{1}{2}}.
\end{equation}
By using Itô's formula for the expected value (c.f. \cite[Remark 4.2.8]{Liu2015}) the product rule and Fubini we obtain for $\lambda \geq 0$ that
\begin{align}\label{IEA}
\begin{split}
\mathbb{E}&[e^{-\lambda t} \Vert X(t) \Vert_H^2 ] - \mathbb{E}[ \Vert \psi(0) \Vert_H^2 ]
\\&=\mathbb{E} \bigg[ \int_0^t e^{-\lambda s} \left( 2 {}_{V^*} \langle Y(s),\bar{X}(s) \rangle_V + \Vert Z(s) \Vert_{L_2(U,H)}^2 - \lambda \Vert X(s)\Vert^2_H \right) \text{d}s \bigg].
\end{split}
\end{align}
Let $\phi \in L^p([-r_0,T] \times \Omega, \text{d}t \otimes \Omega;V)$
such that $\phi(\omega,\cdot) \in C([-r_0,T];H)$ for $P$-a.e. $\omega \in \Omega$
and $\mathbb{E}\left[\Vert \phi _t \Vert _{\mathcal{C}(H)}^2\right]<\infty$ for all $t \in [0,T]$
(This implies just like in \ref{StetRem} that $t \mapsto \mathcal{L}_{\phi _t}$ is continuous.).
By using Itô's formula for the expected value in the case $V=H=V^*=H_{n_k}$ and  defining $\mu^{(n)}_t:=\mathcal{L}_{{X^{(n)}_t}}$, $\nu_t:=\mathcal{L}_{{\phi _t}}$, $t \in [0,T]$, 
we deduce that

\pagebreak
\begin{align}
&\mathbb{E}[e^{-\lambda t} \Vert X^{(n_k)}(t) \Vert_H^2 ]-\mathbb{E}[ \Vert P_{n_k} \psi(0) \Vert_H^2 ] \notag\\
=&\mathbb{E} \bigg[ \int_0^t e^{-\lambda s} \Big( 2 {}_{V^*} \langle P_{n_k} A(s,X^{(n_k)}_s ,\mu^{(n_k)}_s)  ,X^{(n_k)}(s) \rangle_V \notag
\\&\hspace{2cm} + \Vert P_{n_k} B(s,X^{(n_k)}_s ,\mu^{(n_k)}_s) \tilde{P}_{n_k} \Vert_{L_2(U,H)}^2 
- \lambda \Vert  X^{(n_k)}(s)\Vert^2_H \Big) \text{d}s \bigg] \notag\\
\leq & \mathbb{E} \bigg[ \int_0^t e^{-\lambda s} \Big( 2 {}_{V^*} \langle  A(s,X^{(n_k)}_s ,\mu^{(n_k)}_s)  ,X^{(n_k)}(s) \rangle_V \notag
\\&\hspace{2cm} + \Vert  B(s,X^{(n_k)}_s ,\mu^{(n_k)}_s)  \Vert_{L_2(U,H)}^2 
- \lambda \Vert  X^{(n_k)}(s)\Vert^2_H \Big) \text{d}s \bigg] \notag\\
=& \mathbb{E} \bigg[ \int_0^t e^{-\lambda s} \bigg( 2 {}_{V^*} \langle  A(s,X^{(n_k)}_s ,\mu^{(n_k)}_s) -A(s,\phi_s ,\nu_s)  ,X^{(n_k)}(s) - \phi (s) \rangle_V \notag\\
&\hspace{2cm}  + \Vert  B(s,X^{(n_k)}_s ,\mu^{(n_k)}_s) - B(s,\phi_s ,\nu_s) \Vert_{L_2(U,H)}^2 \notag
\\ &\hspace{2cm}- \lambda \Vert X^{(n_k)}(s) - \phi (s) \Vert _H^2
\bigg) \text{d}s \bigg] \label{IAb}\\
& \ + \mathbb{E} \bigg[ \int_0^t e^{-\lambda s} \Big( 2 {}_{V^*} \langle A(s,\phi_s ,\nu_s) ,X^{(n_k)}(s) \rangle _V \notag
\\ &\hspace{65pt} +  2 {}_{V^*} \langle  A(s,X^{(n_k)}_s ,\mu^{(n_k)}_s) -A(s,\phi_s ,\nu_s)  , \phi (s) \rangle_V \notag
\\ &\hspace{65pt} -  \Vert B(s,\phi_s ,\nu_s) \Vert_{L_2(U,H)}^2 
+ 2 \langle  B(s,X^{(n_k)}_s ,\mu^{(n_k)}_s), B(s,\phi_s ,\nu_s) \rangle _{L_2(U,H)} \notag
\\& \hspace{65pt}- 2 \lambda \langle X^{(n_k)}(s), \phi (s) \rangle _H 
+ \lambda \Vert \phi (s) \Vert _H^2 \Big)  \text{d}s \bigg]. \notag
\end{align}
By the definition of the Wasserstein distance, it is $\mathbb{W}_2(\mu^{(n)}_t, \nu _t) \leq \mathbb{E} \big[\Vert X^{(n_k)}_t - \phi _t \Vert _{L^2_H}^2 \big]$ for all $t \in [0,T]$.
This together with Lemma \ref{A1} and (H3) implies for $\lambda:=2 \beta r_0$
\begin{equation}
\begin{split}
&\mathbb{E} \bigg[ \int_0^t e^{-\lambda s} \Big( 2 {}_{V^*} \langle  A(s,X^{(n_k)}_s ,\mu^{(n_k)}_s) -A(s,\phi_s ,\nu_s)  ,X^{(n_k)}(s) - \phi (s) \rangle_V
 \\
& \qquad + \Vert  B(s,X^{(n_k)}_s ,\mu^{(n_k)}_s) - B(s,\phi_s ,\nu_s) \Vert_{L_2(U,H)}^2 
- \lambda \Vert X^{(n_k)}(s) - \phi (s) \Vert _H^2
\Big) \text{d}s \bigg] \\
&\leq
\mathbb{E} \bigg[ \beta \int_0^t e^{-\lambda s} \Big( \Vert X^{(n_k)}_s - \phi _s \Vert _{L^2_H}^2
+ \mathbb{W}_2(\mu^{(n)}_s, \nu _s)
- \lambda \Vert X^{(n_k)}(s) - \phi (s) \Vert _H^2
\Big) \text{d}s \bigg]\\
& \leq \mathbb{E} \bigg[ \int_0^t e^{-\lambda s} \left(2 \beta r_0 - \lambda\right)\Vert X^{(n_k)}(s) - \phi (s) \Vert _H^2\text{d}s \bigg]
+2 \beta r_0e^{\lambda r_0}  \mathbb{E}  \left[\Vert X^{(n_k)}_0 - \phi_0 \Vert _{L_H^2}^2 \right] \\
& \leq 2 \beta r_0e^{\lambda r_0}  \mathbb{E}  \left[\Vert X^{(n_k)}_0 - \phi_0 \Vert _{L_H^2}^2 \right],
\end{split}
\label{HAb}
\end{equation}
for all $t \in [0,T]$.
By inserting (\ref{HAb}) in (\ref{IAb}) and letting $k \rightarrow \infty$, we conclude by (i)-(iii), Fubini's theorem and (\ref{FAb}) 
that for every non-negative  $\rho \in L^{\infty}([0,T], \text{d}t; \mathbb{R})$
\begin{align*}
&\mathbb{E} \bigg[ \int_0^T \rho(t) \lbrace e^{-\lambda t} \Vert X(t) \Vert _H^2 - \Vert \psi (0) \Vert_H^2 \rbrace \text{d}t \bigg] \\
\leq & \mathbb{E} \bigg[ \int_0^T \rho(t) \bigg\lbrace \int_0^t e^{-\lambda s} 
\Big(  2 {}_{V^*} \langle A(s,\phi_s ,\nu_s), \bar{X}(s) \rangle _V 
+  2 {}_{V^*} \langle  Y(s) -A(s,\phi_s ,\nu_s)  , \phi (s) \rangle_V \\
&\hspace{30pt} -\Vert B(s,\phi_s ,\nu_s) \Vert_{L_2(U,H)}^2 
+ 2 \langle  Z(s), B(s,\phi_s ,\nu_s) \rangle _{L_2(U,H)}
-2 \lambda \langle X(s), \phi (s) \rangle _H  \\
&\hspace{30pt} + \lambda \Vert \phi (s) \Vert _H^2 \Big)  \text{d}s \bigg\rbrace  \text{d}t \bigg] 
+  \bigg(\int_0^T \rho(t)\text{d}t\bigg) 2 \beta r_0e^{\lambda r_0}   \mathbb{E}  \left[\Vert X_0 - \phi_0 \Vert _{L_H^2}^2 \right] .
\end{align*}
Inserting (\ref{IEA}) for the left-hand site rearranging and defining 
$L=L(\rho,\beta,r_0,T):=\Big(\int_0^T \rho(t)\text{d}t\Big)
2 \beta r_0e^{\lambda r_0} T$ we arrive at
\begin{equation}
\begin{split}
\mathbb{E} & \bigg[ \int_0^T \rho(t) \bigg\lbrace \int_0^t e^{-\lambda s} 
\Big(2 {}_{V^*} \langle  Y(s) -A(s,\phi_s ,\nu_s)  ,\tilde{X}(s) - \phi (s) \rangle_V\\
&\hspace{55pt} + \Vert B(s,\phi_s ,\nu_s)-Z(s) \Vert_{L_2(U,H)}^2
- \lambda\Vert X(s) - \phi (s) \Vert _H^2 
\Big) \text{d}s \bigg\rbrace \text{d}t \bigg] 
\\ &\leq L  \mathbb{E}  \left[\Vert X_0 - \phi_0 \Vert _{L_H^2}^2 \right]
\end{split}
\label{ABF}
\end{equation}
Taking $\phi = \bar{X}$ and noting that $\bar{X}_0=\psi=X_0$ $P$-a.s. we obtain from (\ref{ABF}) that
\begin{align*}
\mathbb{E} \bigg[ \int_0^T \rho(t) \bigg\lbrace \int_0^t e^{-\lambda s}
\Vert B(s,\bar{X}_s ,\mathcal{L}_{\bar{X}_s})-Z(s) \Vert_{L_2(U,H)}^2 
\bigg\rbrace \text{d}t \bigg] 
\leq 0
\end{align*}
and therefore $B(\cdot,\bar{X}_{\cdot}, \mathcal{L}_{\bar{X}_{\cdot}})=Z \ \  \text{d}t \otimes P-a.e.$ \\
Finally we take $\phi = \bar{X}-\epsilon \tilde{\phi}v$ for $\epsilon>0$,
$v \in V$ and 
$\tilde{\phi} \in L^{\infty}([-r_0,T] \times \Omega, \text{d}t \otimes P;\mathbb{R})$
with $\tilde{\phi}(\omega,\cdot)$ continuous for $P$-a.e $\omega \in \Omega$,
$\mathbb{E}\left[\Vert \tilde{\phi} _t \Vert^2 _{\mathcal{C}(\mathbb{R})} \right]<\infty$ for all $t \in [0,T]$.
(\ref{ABF}) now implies 
\begin{align}
\epsilon \bigg(
\mathbb{E}  \bigg[ &\int_0^T \rho(t) \bigg\lbrace \int_0^t e^{-\lambda s} 
\left(2 {}_{V^*} \langle  Y(s) -A(s,\phi_s ,\nu_s)  ,\tilde{\phi} (s)v \rangle_V
- \epsilon \Vert \tilde{\phi} (s)v \Vert _H^2 \right) \text{d}s \bigg\rbrace \text{d}t \bigg] \bigg) \notag\\
&\leq \epsilon ^2 L \mathbb{E}\left[\Vert \tilde{\phi}v \Vert _{L_H^2}^2 \right].
\label{FE}
\end{align}
Dividing (\ref{FE}) by $\epsilon$ and taking $\epsilon \rightarrow 0$ we obtain by Lebesgue's dominated convergence theorem, (H1) and (H4) that
\begin{align*}
\mathbb{E}  \bigg[ \int_0^T \rho(t) \bigg\lbrace \int_0^t e^{-\lambda s} 
2 {}_{V^*} \langle  Y(s) -A(s,\bar{X}_s ,\mathcal{L}_{\bar{X}_s})  ,\tilde{\phi}(s)v \rangle_V
\text{d}s \bigg\rbrace \text{d}t \bigg] 
\leq 0.
\end{align*}
Replacing $\tilde{\phi}$ with $-\tilde{\phi}$ leads to
\begin{equation*}
\mathbb{E}  \bigg[ \int_0^T \rho(t) \bigg\lbrace \int_0^t e^{-\lambda s} 
2 {}_{V^*} \langle  Y(s) -A(s,\bar{X}_s ,\mathcal{L}_{\bar{X}_s})  ,\tilde{\phi}(s)v \rangle_V
\text{d}s \bigg\rbrace \text{d}t \bigg] = 0.
\end{equation*}
By the arbitrariness of $\rho$, $\tilde{\phi}$ and $ v$ we conclude with that
$A(\cdot,\bar{X}_{\cdot}, \mathcal{L}_{\bar{X}_{\cdot}})=Y$.
This completes the proof of existence.
The uniqueness follows directly from the theorem below.
\end{proof}

\begin{thm}
Consider the situation of Theorem \ref{MT3} and let $X,Y$ be two solutions of (\ref{7}) in the sense of Definition \ref{DS}.
Then for $\beta \geq 0 $ as in (H3) and $t \in [0,T]$:
\begin{itemize}
\item[(i)]
\begin{equation*}
\mathbb{E}\left[\Vert X(t) - Y(t) \Vert _H^2\right] 
\leq \left(1+r_0^2 e^{2 \beta r_0^2 }\right)
 e^{2 \beta r_0t} \mathbb{E}\left[\Vert X_0 - Y_0 \Vert _{\mathcal{C}(H)}^2\right],
\end{equation*}
\item[(ii)]
\begin{align*}
\mathbb{E}&\bigg[\sup_{s \in [0,t]} \Vert X_s -Y_s \Vert_{\mathcal{C}(H)}^2 \bigg]\\
&\leq \inf_{\epsilon \in (0,1)} \Bigg(\Bigg(\frac{\mathbb{E}\Big[ \Vert X_0-Y_0 \Vert_{\mathcal{C}(H)} ^2\Big]}{1-\epsilon} \Bigg)
\exp\left[\frac{2 r_0 }{1-\epsilon}\left( 1+ \frac{6}{\epsilon}  \right)\beta t \right]\Bigg).
\end{align*}
\end{itemize}
\label{Eind3}
\end{thm}
\begin{proof}
(i)
By our definition of solution (Definition \ref{DS}) we can apply Itô's formula
to $X-Y$ and the product rule to obtain for $t \in [0,T]$
\begin{align*}
e^{-2 \beta r_0 t} &\mathbb{E}\Big[ \Vert X(t) - Y(t) \Vert _H^2\Big]
\\ &= \mathbb{E}\left[\Vert X(0) - Y(0) \Vert_H^2 \right] 
\\ &\qquad + \mathbb{E} \bigg[ \int_{0}^t e^{-2 \beta r_0 s}  2 {}_{V^*} \langle A(s,\bar{X}_s,\mathcal{L}_{\bar{X}_s})-A(s,\bar{Y}_s,\mathcal{L}_{\bar{Y}_s}), \bar{X}(s)-\bar{Y}(s) \rangle _V
\\ &\hspace{120pt} + \Vert B(s,\bar{X}_s,\mathcal{L}_{\bar{X}_s})-B(s,\bar{Y}_s,\mathcal{L}_{\bar{Y}_s}) \Vert _{L_2(U,H)}^2  \text{d}s \bigg] 
\\ &\qquad -2 \beta r_0\int_{0}^t e^{-2 \beta r_0 s} \mathbb{E}\left[[\Vert X(s) - Y(s) \Vert_H^2  \right] \text{d}s
\\ &\leq  \mathbb{E}\left[\Vert X(0) - Y(0) \Vert_H^2 \right] +\mathbb{E} \bigg[ \int_{0}^t e^{-2 \beta r_0 s} \Vert X_s - Y_s \Vert _{L_H^2}^2  +\mathbb{W}_2(\mathcal{L}_{\bar{X}_s},\mathcal{L}_{\bar{Y}_s}) \text{d}s \bigg]
\\ &\qquad -2 \beta r_0\int_{0}^t e^{-2 \beta r_0 s} \mathbb{E}\left[[\Vert X(s) - Y(s) \Vert_H^2  \right] \text{d}s
\\ &\leq \mathbb{E}\left[\Vert X(0) - Y(0) \Vert_H^2 \right]
+ r_0 e^{2 \beta r_0^2}  \mathbb{E}  \left[\Vert X_0 - Y_0 \Vert _{L_H^2}^2 \right]
\\ &\leq \left(1+r_0^2 e^{2 \beta r_0^2}\right) \mathbb{E}\left[\Vert X_0 - Y_0 \Vert _{\mathcal{C}(H)}^2\right].
\end{align*}
Where we used (H3) to obtain the first estimate.
To obtain the second estimate we used that by the definition of the Wasserstein distance
$\mathbb{W}_2(\mathcal{L}_{\bar{X}_s},\mathcal{L}_{\bar{Y}_s}) \leq \mathbb{E}[\Vert X_s - Y_s \Vert _{L_H^2}^2]$ together with Lemma \ref{A1}. \\
Multiplying with $e^{2 \beta r_0 t}$ yields (i). \\

(ii)
The proof is similar to the proof of Theorem \ref{MT2} (b).
\end{proof}

Note that in the proof of Theorem \ref{Eind3} (i) only the first estimate in (H3) is used.
In fact this is only needed to show the existence of solutions to the finite dimensional equation (\ref{9}).
That means that if it were possible to proof the existence of solutions to finite dimensional path-distribution dependent SDE's without the ``Lipschitz-condition''
\begin{align*}
\int_{0}^t \Vert \sigma(s,\xi_s, \mu_s)-\sigma(s,\eta _s, \nu _s) \Vert _{\text{HS}}^2 
\leq \beta (t) \int_{0}^t  \Vert \xi _s- \eta _s \Vert _{\infty}^2
+ \mathbb{W}_2(\mu _s , \nu _s)^2 \text{d}s
\end{align*}
in (H3) in chapter 2, the second part of (H3) in this chapter could be dropped and we would have existence and uniqueness of solutions to (\ref{7}) in the sense of Definition \ref{DS} by the arguments presented in this chapter. \\
\\
By Theorem \ref{MT3} we know that $\mathbb{E}[\sup_{t \in [-r_0,T]} \Vert X(t) \Vert _H ^2 ] < \infty$.
The final proposition of this chapter gives a more precise estimate of  $\mathbb{E}[\sup_{t \in [-r_0,T]} \Vert X(t) \Vert _H ^2 ]$. The proof is analogous to the proof of Theorem \ref{MT2} (b) (ii).
\begin{prop}
Consider the situation of Theorem \ref{MT3} and let $X,Y$ be two solutions of (\ref{7}) in the sense of Definition \ref{DS}.
Then
\begin{align*}
\mathbb{E}\bigg[\sup _{r \in [-r_0,T]}\Vert X(r) \Vert^2 _H
+\int_0^T \Vert X(s) \Vert ^p _V\text{d}s \bigg] 
\leq C\left(1 + \mathbb{E}\left[\Vert X_0 \Vert^2_{\mathcal{C}(H)}\right] \right)
\end{align*}
for some $C>0$.
\end{prop}

\begin{proof}
The proof is essentially the same as the proof of Theorem \ref{MT2} (b) (ii).
\end{proof}

\section*{Acknowledgment}
I want to thank Prof. Dr. Michael Röckner and Dr. Tatjana Pasurek for fruitful discussion and helpful comments during the preparation of this paper.
Support of the DFG through CRC 1283 is also gratefully acknowledged.

\end{document}